\documentclass[11pt]{amsart}
\usepackage{graphicx}
\usepackage{latexsym,amsmath,amsfonts,amscd, amsthm}
\usepackage{color}
\usepackage[colorlinks=true]{hyperref}
\hypersetup{urlcolor=blue, citecolor=red}
\usepackage{tikz}
\usetikzlibrary{arrows,backgrounds,snakes}
\usepackage{epsfig}
\usepackage{amssymb}
\usepackage{xcolor}
\usepackage{comment}
\usepackage{subfigure}
\usepackage{multirow}
\newtheoremstyle{plainNoItalics}{}{}{\normalfont}{}{\bfseries}{.}{ }{}
\theoremstyle{plain}
\newtheorem{thm}{Theorem}[section]

\newtheorem{defn}[thm]{Definition}
\newtheorem{rem}[thm]{Remark}
\newtheorem{prop}[thm]{Proposition}

\newcommand{\f}{\frac}

\newcommand{\beq}{\begin{equation}}
\newcommand{\eeq}{\end{equation}}
\newcommand{\beqa}{\begin{eqnarray}}
\newcommand{\eeqa}{\end{eqnarray}}
\newcommand{\bit}{\begin{itemize}}
\newcommand{\eit}{\end{itemize}}
\newcommand{\bedef}{\begin{defn}}
\newcommand{\edefn}{\end{defn}}
\newcommand{\bpro}{\begin{prop}}
\newcommand{\epro}{\end{prop}}
\newcommand{\RR}{\mathbb{R}}
\newcommand{\NN}{\mathbb{N}}

\newcommand\bx{{\mathbf x}}
\newcommand\bv{{\mathbf v}}
\newcommand\bE{{\mathbf E}}
\DeclareMathOperator\dD{d}
\newcommand{\mE}{\mathcal E}
\newcommand{\mH}{\mathcal H}

\newcommand{\sumn}{\sum_{n=0}^{N_H-1}}
\newcommand{\sumj}{\sum_{j=-N_x}^{N_x}}

\newcommand{\bC}{{\mathbf C}}

\newcommand{\xL}{{x_{i-\frac{1}{2}}}}
\newcommand{\xR}{{x_{i+\frac{1}{2}}}}

\newcommand{\iR}{{i+\frac{1}{2}}}

\newcommand{\testR}{{\varphi}}   
\newcommand{\testP}{{\eta}}   
\newcommand{\testE}{{\zeta}}

\newcommand{\pf}{\partial f}
\newcommand{\pt}{\partial t}
\newcommand{\px}{\partial x}

\newcommand{\pv}{\partial v}

\newcommand\ds{ \displaystyle }

\email{francis.filbet@math.univ-toulouse.fr}
\email{marianne.bessemoulin@univ-nantes.fr}

\title[Stable and conservative methods for the Vlasov-Poisson system]{On the stability of conservative discontinuous Galerkin/Hermite Spectral methods for the  Vlasov-Poisson System} 

\author{Marianne Bessemoulin-Chatard and Francis Filbet}



\begin{document}

\maketitle

\centerline{\scshape Marianne Bessemoulin-Chatard}
\medskip
{\footnotesize
    \centerline{Laboratoire de Mathématiques Jean Leray, Universit\'e de Nantes,}
    \centerline{2, rue de la Houssinière}
    \centerline{BP 92208, F-44322 Nantes Cedex 3, France}
  }

\medskip

\centerline{\scshape Francis Filbet}
\medskip
{\footnotesize
    \centerline{Institut de Math\'ematiques de Toulouse , Universit\'e Paul Sabatier}
    \centerline{Toulouse, France}
}

\bigskip

\begin{abstract}
We study a class of spatial discretizations for the Vlasov-Poisson system written as an hyperbolic system using Hermite polynomials. In particular, we focus on spectral methods and discontinuous Galerkin approximations. To obtain $L^2$ stability properties, we introduce a new $L^2$ weighted space, with a time dependent weight. For the Hermite spectral form of the Vlasov-Poisson system, we prove conservation of mass, momentum and total energy, as well as global stability for the weighted $L^2$ norm. These properties are then discussed for several spatial discretizations. Finally, numerical simulations are performed with the proposed DG/Hermite spectral method to highlight its stability and conservation features.
\end{abstract}

\vspace{0.1cm}

\tableofcontents

\section{Introduction}
\setcounter{equation}{0}
\setcounter{figure}{0}
\setcounter{table}{0}

One of the simplest model  that is currently applied in plasma physics
simulations is the Vlasov-Poisson system. This system
describes the evolution  of charged particles in the case where the only interaction considered is the mean-field force created through electrostatic effects. The system consists in Vlasov equations for phase space density  $f_\beta(t,\bx,\bv)$ of each particle species $\beta$ of charge $q_\beta$ and mass~$m_\beta$
\begin{equation}
  \label{vlasov0}
  \left\{
    \begin{array}{l}
\ds\frac{\partial f_\beta}{\partial t}\,+\,\bv\cdot\nabla_\bx f_\beta \,+\,\frac{q_\beta}{m_\beta}\bE\cdot\nabla_\bv f_\beta \,=\, 0\,,
      \\[1.1em]
      f_\beta(t=0) = f_{\beta,0}\,,
      \end{array}\right.
      \end{equation}
coupled to its self-consistent electric field $\bE= -\nabla_\bx \Phi$ which satisfies the Poisson equation
\begin{equation}
\label{poisson0}
-4\pi\,\epsilon_0\,\Delta_\bx \Phi \,=\, \sum_{\beta} q_\beta
n_\beta\,,\qquad{\rm with }\quad n_\beta = \int_{\RR^d} f_\beta \,\dD\bv\,,
\end{equation}
where $\epsilon_0$ is the vacuum permittivity. On the one hand, for a smooth and nonnegative initial data $f_{\beta,0}$, the solution $f_\beta(t)$
to \eqref{vlasov0} remains smooth and nonnegative for all $t\geq
0$.  On the other hand, for any function  $G\in C^1(\RR^+,\RR^+)$, we have
$$
\frac{\dD}{\dD t}\int_{\RR^{2d}} G(f_\beta(t)) \,\dD \bx \dD\bv
=0,\quad\forall t\in\RR^+\,,
$$
which leads to the conservation of mass,  $L^p$ norms, for $1 \leq p \leq
+\infty$ and  kinetic entropy, 
$$
\mH(t) \,:=\, \int_{\RR^{2d}}
f_\beta(t)\,\ln\left(f(t)\right)\dD\bx \dD\bv \,=\, \mH(0), \quad\forall
t\geq 0\,.
$$
 We also get the conservation of momentum
$$
\sum_{\beta} \int_{ \RR^{2d}} m_\beta \, \bv \,
f_\beta(t)\,\dD\bx\dD\bv =  \sum_{\beta} \int_{\RR^{2d}} m_\beta \, \bv \,
f_{\beta,0}\,\dD\bx\dD\bv 
$$
and total energy
$$
\mE(t)\, :=\, \sum_{\beta}\frac{m_\beta}{2}\int_{\RR^{2d}}
f_\beta(t) \|\bv\|^2 \dD\bx \dD\bv \,+\, 2\pi\epsilon_0 \int_{\RR^d}
\|\bE\|^2 \dD\bx   \,=\, \mE(0)\,, \quad\forall
t\geq 0\,.
$$


Numerical approximation of the Vlasov-Poisson system has been addressed since the sixties. Particle methods (PIC), consisting in approximating the plasma by a finite number of macro particles, have been widely used \cite{Birdsall1985}. They allow to obtain satisfying results with a few number of particles, but a well-known drawback of this class of methods is their inherent numerical noise which only decreases in $1/\sqrt{N}$ when the number of particles $N$ increases, preventing from getting an accurate description of the distribution function for some specific applications. To overcome this difficulty, Eulerian solvers, that is methods discretizing the Vlasov equation on a mesh of the phase space, can be considered. Many authors explored their design, and an overview of many different techniques with their pros and cons can be found in \cite{Filbet2003}. Among them, we can mention finite volume methods \cite{Filbet2001} which are a simple and inexpensive option, but in general low order. Fourier-Fourier transform schemes \cite{Klimas1994} are based on a Fast Fourier Transform of the distribution function in phase space, but suffer from Gibbs phenomena if other than periodic conditions are considered. Standard finite element methods \cite{Zaki1988a,Zaki1988b} have also been applied, but may present numerical oscillations when approximating the Vlasov equation. Later, semi-Lagrangian schemes have also been proposed \cite{Sonnendrucker1999}, consisting in computing the distribution function at each grid point by following the characteristic curves backward. Despite these schemes can achieve high order allowing also for large time steps, they require high order interpolation to compute the origin of the characteristics, destroying the local character of the reconstruction.

In the present article, using Hermite polynomials in the velocity variable, we write the Vlasov-Poisson system \eqref{vlasov0}--\eqref{poisson0} as an hyperbolic system. This idea of using Galerkin methods with a small finite set of orthogonal polynomials rather than discretizing the distribution function in velocity space goes back to the 60's \cite{Armstrong1967,Joyce1971}. More recently, the merit to use rescaled orthogonal basis like the so-called scaled Hermite basis has been shown \cite{engelmann1963, Holloway1996,holloway2,Schumer1998,Tang1993}. In \cite{Holloway1996}, Holloway formalized two possible approaches. The first one, called symmetrically-weighted (SW), is based on standard Hermite functions as the basis in velocity and as test functions in the Galerkin method. It appears that this SW method cannot simultaneously conserve mass and momentum. It makes up for this deficiency by correctly conserving the $L^2$ norm of the distribution function, ensuring the stability of the method. In the second approach, called asymmetrically-weighted (AW), another set of test functions is used, leading to the simultaneous conservation of mass, momentum and total energy. However, the AW Hermite method does not conserve the $L^2$ norm of the distribution function and is then not numerically stable. In addition, the asymmetric Hermite method exactly solves the spatially uniform problem without any truncation error, a property not shared by either the traditional symmetric Hermite expansion or by finite difference methods. The aim of this work is to present a class of numerical schemes based on the AW Hermite methods and to provide a stability analysis.

In what follows, we consider two types of spatial discretizations for the Vlasov-Poisson system \eqref{vlasov0}-\eqref{poisson0}, written as an hyperbolic system using Hermite polynomials in the velocity variable: spectral methods and a discontinuous Galerkin (DG) method. Concerning spectral methods, the Fourier basis is the natural choice for the spatial discretization when considering periodic boundary conditions. Spectral Galerkin and spectral collocation methods for the AW Fourier-Hermite discretization have been proposed in \cite{engelmann1963, LeBourdiec2006, manzini2016}. In \cite{Camporeale2016}, authors study a time implicit AW Fourier-Hermite method allowing exact conservation of charge, momentum and energy, and highlight that for some test cases, this scheme can be significantly more accurate than the PIC method. 

For the SW Fourier-Hermite method, a convergence theory was proposed in \cite{Manzini2017}. In \cite{Kormann2021}, authors study conservation and $L^2$ stability properties of a generalized Hermite-Fourier semi-discretization, including as special cases the SW and AW approaches. Concerning discontinuous Galerkin methods, they are similar to finite elements methods but use discontinuous polynomials and are particularly well-adapted to handling complicated boundaries which may arise in many realistic applications. Due to their local construction, this type of methods provides good local conservation properties without sacrificing the order of accuracy. They were already used for the Vlasov-Poisson system in \cite{Heath2012,Cheng2013}. Optimal error estimates and study of the conservation properties of a family of semi-discrete DG schemes for the Vlasov-Poisson system with periodic boundary conditions have been proved for the one \cite{Ayuso2011} and multi-dimensional \cite{Ayuso2012} cases. In all these works, the DG method is employed using a phase space mesh.

Here, we adopt this approach only in physical space, as in \cite{Filbet2020}, with a Hermite approximation in the velocity variable. In \cite{Filbet2020}, such schemes with discontinuous Galerkin spatial discretization are designed in such a way that conservation of mass, momentum and total energy is rigorously provable.

As mentionned before, the main difficulty is to study the stability of approximations based on asymmetrically-weighted Hermite basis. Indeed, this choice fails to preserve $L^2$ norm of the approximate solution, and therefore to ensure long-time stability of the method. In this framework, the natural space to be considered is
\[L^2_\omega:=\left\{u : \RR^d\times \RR^d\to\RR : \int_{\RR^{2d}} |u(\bx,\bv)|^2\,(2\pi)^{d/2}e^{\|\bv\|^2/2}\dD\bx \dD\bv <+\infty\right\},\]
and there is no estimate of the associated norm for the solution to the Vlasov-Poisson system~\eqref{vlasov0}-\eqref{poisson0}. To overcome this difficulty, we introduce a new $L^2$ weighted space, with a time-dependent weight, allowing to prove global stability of the solution in this space. Actually, this idea has been already employed in \cite{Ma2005,Ma2007} to stabilize Hermite spectral methods for linear diffusion equations and nonlinear convection-diffusion equations in unbounded domains, yielding stability and spectral convergence of the considered methods. Here, we define the weight as 
\begin{equation}\label{eq:def_omega}
\omega(t,\bv)\;:=\,(2\pi)^{d/2}\,e^{(\alpha(t)\,\|\bv\|)^2/2},
\end{equation}
where $\alpha$ is a nonincreasing positive function which will be designed in such a way that a global stability estimate can be established in the following $L^2$ weighted space:
\begin{equation*}
L^2_{\omega(t)}:=\left\{u : \RR^{2d}\to\RR : \int_{\RR^{2d}} |u(\bx,\bv)|^2\omega(t,\bv) \dD\bx \dD\bv <+\infty\right\},
\end{equation*}
with $\|\cdot\|_{\omega(t)}$ the corresponding norm, that is
\begin{equation*}
\|u\|_{\omega(t)}^2=(2\pi)^{d/2}\int_{\RR^{2d}}|u(\bx,\bv)|^2\,e^{(\alpha(t)\|\bv\|)^2/2}\,\dD\bx\,\dD\bv.
\end{equation*}
Let us now determine the function $\alpha$. To do so, we compute the time derivative of $ \|f_\beta(t)\|_{\omega(t)}$, $f_\beta$ being the solution of \eqref{vlasov0}. Using the Vlasov equation \eqref{vlasov0} and the definition \eqref{eq:def_omega} of the weight $\omega(t,\bv)$, one has
\begin{align*}
\frac{1}{2}\frac{\dD}{\dD t}\|f_\beta(t)\|_{\omega(t)}^2= & -\int_{\RR^{2d}}f_\beta\left(\bv\cdot \nabla_\bx f_\beta+\frac{q_\beta}{m_\beta}\bE\cdot\nabla_\bv f_\beta\right)\omega \dD\bx \dD\bv 
\\
& +\frac{1}{2}\int_{\RR^{2d}}\alpha\,\alpha'\, \|\bv\|^2\,f_\beta^2\,\,\omega \dD\bv \dD\bx.
\end{align*}
Then, since 
\[ \int_{\RR^d}f_\beta\,\bE\cdot\nabla_\bv f_\beta\,\omega \,d\bv = -\frac{1}{2}\int_{\RR^d} \alpha^2f_\beta^2\,\bE\cdot \bv\,\omega\,\dD\bv,\]
we obtain
\begin{equation*}
\frac{1}{2}\frac{\dD }{\dD t}\|f_\beta(t)\|_{\omega(t)}^2\;=\,\frac{1}{2}\int_{\RR^{2d}}f_\beta^2\left(\frac{q_\beta}{m_\beta}\alpha^2\bE\cdot\bv+\alpha\alpha'\|\bv\|^2\right)\omega \dD\bx \dD\bv.
\end{equation*}
Applying now Young inequality on the first term, we get for $\gamma >0$,
\begin{equation*}
\frac{1}{2}\frac{\dD }{\dD t}\|f_\beta(t)\|_{\omega(t)}^2\;\leq\,\frac{1}{2}\int_{\RR^{2d}}f_\beta^2\left(\frac{\gamma}{2}\frac{q_\beta^2}{m_\beta^2}\alpha^4\,\|\bE\|_\infty^2\,\|\bv\|^2+\frac{1}{2\,\gamma}+\alpha\alpha'\|\bv\|^2\right)\omega \dD\bx \dD\bv.
\end{equation*}
We notice that if $\alpha$ is such that 
\begin{equation}\label{design_alpha}
\alpha'=-\frac{\gamma}{2}\frac{q_\beta^2}{m_\beta^2}\|\bE(t)\|_{\infty}^2\,\alpha^3=:I(\alpha,E),
\end{equation}
the first and third terms cancel each other, yielding
\begin{equation*}
\frac{1}{2}\frac{\dD}{\dD t}\|f_\beta(t)\|_{\omega(t)}^2 \leq  \frac{1}{4\,\gamma}\|f_\beta(t)\|_{\omega(t)}^2.
\end{equation*}
Finally, applying Gr\"onwall's Lemma gives
\[\|f_\beta(t)\|_{\omega(t)}\leq \|f_{\beta,0}\|_\omega\, e^{t/4\gamma}.\]
It now remains to define $\alpha$ satisfying \eqref{design_alpha}. We remark that 
\begin{equation}
  \label{eq:def_alpha}
\alpha(t)\,:=\,\alpha_0\,\left(1+\gamma\,\frac{q_\beta^2}{m_\beta^2}\int_0^t\|\bE(s)\|^2_\infty \dD s\right)^{-1/2}
\end{equation}
is a suitable choice, where $\alpha_0$ is the initial value
of $\alpha$ at $t=0$ and corresponds to the initial scale of the
distribution function, whereas $\gamma$ is a free parameter. This function is positive, nonincreasing and the parameter $\gamma>0$ will practically be chosen sufficiently small to ensure that $\alpha(t)$ does not decrease too fast towards $0$ as $t$ goes to infinity.\\

Summarizing, we have established the following result.

\begin{prop}
  \label{prop:L2cont}
Assuming that the initial data $f_{\beta,0}$ belongs to $L^2_{\omega(0)}=L^2_{\omega_0}$, then the solution $f_\beta(t)$ to \eqref{vlasov0} satisfies for all $t\geq 0$:
\begin{equation*}
\|f_\beta(t)\|_{\omega(t)}\leq \|f_{\beta,0}\|_{\omega_0}\, e^{t/4\gamma},
\end{equation*}
where $\gamma>0$ is the constant appearing in the definition \eqref{eq:def_alpha} of $\alpha$.
\end{prop}

Notice that the weight depends on the solution to the Vlasov-Poisson, hence it is mandatory to control the $L^\infty$ norm of $\bE$ in order that this estimate makes sense. In the next section, we introduce the formulation of the Vlasov equation using the Hermite basis in velocity, and prove the analogous of Proposition \ref{prop:L2cont} for the obtained system. Then in Section \ref{sec:3}, we discuss conservation and stability properties for a class of spatial discretizations including Fourier spectral method and discontinuous Galerkin approximations. In Section \ref{sec:4}, we introduce the time discretization that will be used to compute numerical results with the discontinuous Galerkin method. Finally in Section \ref{sec:5} we present numerical results for two stream instability, bump-on-tail problem and ion acoustic wave to highlight conservations and stability of our discretization. 

%
\section{Hermite spectral form of the Vlasov equation}\label{sec:hermite}
\setcounter{equation}{0}
\setcounter{figure}{0}
\setcounter{table}{0}

For simplicity, we now set all the physical constants to one and consider the one dimensional Vlasov-Poisson
system for a single species with periodic boundary conditions in space, where the Vlasov equation reads
\beq
\label{vlasov}
\f{\pf}{\pt}\,+\, v\,\f{\pf}{\px} \,+\,E\,\f{\pf}{\pv} \,\,=\,\, 0\,
\eeq
with $t\geq 0$, position $x \in (0,L)$ and velocity
$v\in\RR$. Moreover, the self-consistent electric field $E$  is determined by the Poisson equation
\beq
\label{poisson}
\f{\partial E}{\px}=\rho- \rho_0\,,
\eeq
where the density $\rho$ is given by
$$
\rho(t,x)\,=\,\int_\RR f(t,x,v)\,\dD v\,, \quad t\geq 0, \, x\in (0,L)
$$
and $\rho_0$ is a constant  ensuring the quasi-neutrality condition
of the plasma
\begin{equation*}
\int_{0}^L \left(\rho - \rho_0\right) \,\dD x = 0\,. 
\end{equation*}

\subsection{Hermite basis}

We approximate the solution $f$ of \eqref{vlasov}--\eqref{poisson} by a finite sum which corresponds to a truncation of a
series
\beq
\label{fseries}
f_{N_H}(t,x,v)=\sum_{n=0}^{N_H-1}C_n(t,x)\,\Psi_n(t,v)\,,
\eeq
where $N_H$ is the number of modes. The issue is then to
determine a well-suited class of basis functions $\Psi_n$ and to find
the expansion coefficients $C_n$. Here, our aim is to obtain a control on a $L^2$ norm of $f_N$, in the spirit of that established in Proposition \ref{prop:L2cont}. We then choose the following basis of normalized scaled time dependent asymmetrically weighted Hermite functions:
\beq
\label{hbasisf}
\Psi_n(t,v)\,=\,\alpha(t)\,{H}_n\left(\alpha(t)v\right)\,\f{e^{-(\alpha(t)v)^2/2}}{\sqrt{2\pi}}\,,
\eeq
where $\alpha$ is the positive nonincreasing function given by
\begin{equation}
  \label{defalpha}
\alpha(t)\,\,=\,\,\alpha_0\,\left(1+\gamma\,\int_0^t\|E_{N_H}(s)\|_\infty^2 \dD s\right)^{-1/2},
\end{equation}
with $E_{N_H}$ an approximation of the electric field that will be defined below. Functions $H_n$ are the Hermite polynomials defined by ${H}_{-1}(\xi)=0$, ${H}_0(\xi)=1$ and for $n\geq
1$, ${H}_n(\xi)$ has the following recursive relation 
$$
\sqrt{n}\,{H}_n(\xi) \,=\, \xi \,{H}_{n-1}(\xi)-\sqrt{n-1}\,{H}_{n-2}(\xi)\,, \quad \forall \,n \geq 1\,.
$$
Let us also emphasize that $H_n'(\xi)=\sqrt{n}H_{n-1}(\xi)$ for all $n\geq 1$, and that the Hermite basis \eqref{hbasisf} satisfies the following orthogonality property
\begin{equation}
  \label{proportho}
\int_{\RR}\Psi_{n}(t,v)H_m(\alpha(t)v) \dD v=\delta_{n,m},
\end{equation}
where $\delta_{n,m}$ is the Kronecker delta function.

Now the objective is to obtain an evolution equation for each mode $C_n$ by substituting the expression \eqref{fseries} for $f_{N_H}$ into the Vlasov equation \eqref{vlasov} and using the orthogonality property \eqref{proportho}. Thanks to the definition \eqref{hbasisf} of $\Psi_n$ and the properties of $H_n$, we compute the different terms of \eqref{vlasov}. The time derivative term is given by
\begin{equation*}
\partial_t f_{N_H}=\sum_{n=0}^{N_H-1}\left(\partial_t C_n \Psi_n-C_n\frac{\alpha'}{\alpha}\left(n\Psi_n+\sqrt{(n+1)(n+2)}\,\Psi_{n+2}\right)\right),
\end{equation*}
the transport term is
\begin{equation*}
v\,\partial_xf_{N_H}=\sum_{n=0}^{N_H-1}\partial_x C_n \frac{1}{\alpha}\left(\sqrt{n+1}\,\Psi_{n+1}+\sqrt{n}\,\Psi_{n-1}\right),
\end{equation*}
and finally the nonlinear term is
\begin{equation*}
E_{N_H}\,\partial_v f_{N_H}=-\sum_{n=0}^{N_H-1}E_{N_H} \,C_n\,\alpha\,\sqrt{n+1}\,\Psi_{n+1}.
\end{equation*}
Then taking $H_n(\alpha\,v)$ as test function and using orthogonality property \eqref{proportho}, we obtain an evolution equation for each mode $C_n$, $n=0,\ldots,N_H-1$:
\begin{equation}
  \label{cn}
\partial_t C_n-\frac{\alpha'}{\alpha}\left(nC_n+\sqrt{(n-1)n}\,C_{n-2}\right)+\frac{1}{\alpha}\left(\sqrt{n}\,\partial_x C_{n-1}+\sqrt{n+1}\,\partial_xC_{n+1}\right)-E_{N_H}\,\alpha\,\sqrt{n}C_{n-1}=0,
\end{equation}
with the understanding that $C_{n}=0$ for $n<0$ and $n\geq N_H$.
Meanwhile, we first observe that the density $\rho_{N_H}$ satisfies
$$
\rho_{N_H} \,=\, \int_\RR f_{N_H} \dD v \,=\, C_0\,,
$$
and then the Poisson equation becomes 
\beq
\label{ps}
\f{\partial E_{N_H}}{\px} = C_0 - \rho_0\,.
\eeq
Observe that if we take $N_H=\infty$ in the expression \eqref{fseries},
we get an infinite system \eqref{cn}-\eqref{ps} of equations for $(C_n)_{n\in\NN}$ and $E$, which is
formally equivalent to the Vlasov-Poisson  system
\eqref{vlasov}-\eqref{poisson}.\\

In what follows, we rather consider a generic weak formulation of \eqref{cn}--\eqref{ps}. Indeed, this will allow us to straightforwardly apply the obtained results to spatial discretizations with spectral methods (see Section \ref{sec:Fourier}). 

Let $V$ be a subspace of $H^1(0,L)$. We look for $C_n(t,\cdot) \in V$ such that $\partial_t C_n(t,\cdot)\in V$ and for all $\varphi_n \in V$, 
\begin{align}
\frac{\dD}{\dD t}\int_{0}^LC_{n}\,\varphi_n\,\dD x&-\frac{\alpha'}{\alpha}\int_{0}^L\left(n\,C_{n}+\sqrt{(n-1)n}\,C_{n-2}\right)\varphi_n\,\dD x\nonumber\\
&- \frac{1}{\alpha}\int_0^L \left(\sqrt{n}\, C_{n-1}+\sqrt{n+1}\, C_{n+1}\right)\varphi_n'\,\dD x \label{fvcn} \\
&-\alpha\sqrt{n}\int_{I_j}E_{N_H}\,C_{n-1}\,\varphi_n\,\dD x=0,\quad 0\leq n \leq N_H-1, \nonumber
\end{align}
and for a couple $\Phi_{N_H}(t,\cdot),\,E_{N_H}(t,\cdot)\in V$ such that $\partial_t E_{N_H}(t,\cdot)\in V$, and for all $\eta$ and $\zeta \in V$, we have
\begin{equation}
\label{fvps}
\left\{
\begin{array}{l}
\ds \int_0^L \Phi_{N_H}\,\eta'\,\dD x =\int_0^L E_{N_H}\,\eta\,\dD x,\\ \, \\
\ds-\int_0^L E_{N_H} \,\zeta'\, \dD x = \int_0^L (C_0-\rho_0)\,\zeta\,\dD x.
\end{array}
\right.
\end{equation}

In the rest of this section, we consider $V=H^1(0,L)$, and then \eqref{fvcn}--\eqref{fvps} is the weak formulation of \eqref{cn}--\eqref{ps}.

\subsection{Conservation properties}
%
%
\begin{prop}
  \label{prop:1}
  For any $N_H\in \NN$, consider the distribution function $f_{N_H}$ given
  by the truncated series
  $$
f_{N_H}(t,x,v)=\sum_{n=0}^{N_H-1}C_n(t,x)\Psi_n(t,v)\,,
$$
where $(C_n,E_{N_H})$ is a solution to the Vlasov-Poisson system written
as \eqref{fvcn}-\eqref{fvps}. Then
mass, momentum and total energy are conserved, that is, 
$$
\frac{\dD }{\dD t}\int_0^L \frac{C_k}{\alpha^k} \, \dD x \, =\, 0, \quad k=0, \,1,
$$
and for the total energy,   
$$
\mE_{N_H}(t) \,:=\, \frac{1}{2}\,\int_{0}^L \left(\frac{1}{\alpha^2} {\left(\sqrt{2} \,C_2
 + C_0 \right)} \,+\, |E_{N_H}|^2 \right)\dD x   \,=\, \mE_{N_H}(0)\,.
$$
\end{prop}
\begin{proof}
We consider  the first three equations of \eqref{fvcn}: for all $\varphi_0,\,\varphi_1,\,\varphi_2 \in V$, 
\begin{equation}
  \label{c0}
  \left\{
    \begin{array}{ll}
\ds\frac{\dD }{\dD t} \int_0^L C_0\,\varphi_0\,\dD x & \ds-\frac{1}{\alpha}\int_0^L C_1\,\varphi_0'\,\dD x \,=\, 0\,, \\[0.9em]
	\ds \frac{\dD }{\dD t}\int_0^L C_1\,\varphi_1\,\dD x & \ds-\frac{\alpha'}{\alpha}\int_0^L C_1\,\varphi_1\,\dD x-\frac{1}{\alpha}\int_0^L\left( C_0+\sqrt{2} C_2\right)\,\varphi_1'\,\dD x\\
	& \ds -\,\alpha\int_0^L E_{N_H}\, C_0\,\varphi_1\,\dD x \,=\, 0\,,  \\[0.9em]
\ds	\frac{\dD }{\dD t}\int_0^L C_2\,\varphi_2\,\dD x & \ds -\frac{\alpha'}{\alpha}\int_0^L\left(2C_2+\sqrt{2}C_0\right)\,\varphi_2\,\dD x-\frac{1}{\alpha}\int_0^L\left(\sqrt{2} C_1+\sqrt{3} C_3\right)\varphi_2'\, \dD x\\
& \ds -\sqrt{2}\,\alpha\int_0^L E_{N_H}\,C_1\,\varphi_2\,\dD x \,=\, 0\,,
    \end{array}
  \right.
  \end{equation}
which will be related to the conservation of mass, momentum and
energy. Thus, let us look at the conservation properties from
\eqref{c0}.

First the total mass is defined as
$$
\int_0^L\int_{\RR}f_{N_H}(t,x,v)\,\dD v\,\dD x \,=\,\int_0^L C_0(t,x)\dD
x\,,
$$
hence the conservation of mass directly comes from \eqref{c0} by
taking $\varphi_0= 1$ as test function in the first equation.

Then the momentum is given by
$$
\int_0^L\int_{\RR}v\,f_{N_H}(t,x,v)\,\dD v\,\dD x =\int_0^L
\frac{C_1(t,x)}{\alpha(t)}\,\dD x\,.
$$

We have 
\[\f{\dD }{\dD t}  \int_0^L \frac{C_1}{\alpha}\,\dD
x=\,\int_0^L \left(\frac{\partial_t C_1}{\alpha}-\frac{\alpha'}{\alpha^2}C_1\right)\,\dD x,\]
which taking $\varphi_1=\frac{1}{\alpha}$ in the second equation of \eqref{c0} yields
\[ \f{\dD }{\dD t}  \int_0^L \frac{C_1}{\alpha}\,\dD
x= \int_0^L E_{N_H}\, C_0\, \dD x.\]
Now, choosing $\zeta=E$ in the second equation of \eqref{fvps} gives
\[ \f{\dD }{\dD t}  \int_0^L \frac{C_1}{\alpha}\,\dD
x = \rho_0\int_0^L E_{N_H}\,\dD x-\int_0^L E_{N_H} \,\partial_x E_{N_H}\, \dD x = \rho_0\int_0^L E_{N_H} \,\dD x =0,\] 
which gives the conservation of momentum.

Finally to prove the conservation of total energy $\mathcal{E}_{N_H}$, we compute the
variation of the kinetic energy $\mathcal{E}^K_{N_H}$ defined as
$$
\mathcal{E}^K_{N_H}(t)=\f12\int_0^L\int_{\RR}v^2 f_{N_H}(t,x,v)\,\dD x\,\dD
v\,=\,\frac{1}{2}\,\frac{1}{\alpha^2}\int_{0}^L  {\left(\sqrt{2} \,C_2
 + C_0 \right)}\,\dD x.
$$
We have
\[\frac{\dD \mathcal{E}_{N_H}^K(t)}{\dD t}=\frac{1}{2\,\alpha^2}\int_0^L (\sqrt{2}\,\partial_t C_2+\partial_t C_0)\,\dD x-\frac{\alpha'}{\alpha^3}\int_0^L(\sqrt{2}\,C_2+C_0)\,\dD x.   \]
Thus, combining the first equation of \eqref{fvcn} with $\varphi_0=\frac{1}{2\alpha^2}$ and the third equation of \eqref{fvcn} with $\varphi_2=\frac{1}{\sqrt{2}\,\alpha^2}$, we get
\begin{equation}
\label{bruges}
\frac{\dD \mathcal{E}_{N_H}^K(t)}{\dD t}=\frac{1}{\alpha}\int_0^L E_{N_H}\, C_1\,\dD x .
\end{equation}
Finally, taking $\eta=C_1$ in the first equation of \eqref{fvps}, we obtain
\[\frac{\dD \mathcal{E}_{N_H}^K(t)}{\dD t}=\frac{1}{\alpha}\int_0^L\Phi_{N_H}\,\partial_x C_1\,\dD x.\]
We now compute the time derivative of the potential energy defined by
\[ \mathcal{E}_{N_H}^P(t):= \frac{1}{2}\int_0^L |E_{N_H}|^2\,\dD x.\]
Using $\eta=\partial_t E_{N_H} \in V$ as test function in the first equation of \eqref{fvps} and an integration by parts, we have
\[\frac{\dD \mathcal{E}_{N_H}^P(t)}{\dD t}=\int_0^L\partial_t E_{N_H}\,E_{N_H}\,\dD x=\int_0^L \Phi_{N_H}\,\partial_x\partial_t E_{N_H}\,\dD x=-\int_0^L\partial_x\Phi_{N_H}\,\partial_t E_{N_H}\,\dD x.\]
Now, choosing $\zeta=\Phi$ in the time derivative of the second equation of \eqref{fvps}, and then $\varphi_0=\Phi_{N_H}$ in the first equation of \eqref{fvcn}, we finally get
\[\frac{\dD \mathcal{E}_{N_H}^P(t)}{\dD t}=\int_0^L \partial_t C_0\,\Phi_{N_H}\,\dD x=\frac{1}{\alpha}\int_0^L C_1\,\partial_x\Phi_{N_H}\,\dD x=-\frac{1}{\alpha}\int_0^L\partial_x C_1\,\Phi_{N_H}\,\dD x=-\frac{\dD \mathcal{E}_{N_H}^K(t)}{\dD t}.\]
This concludes the proof since the total energy $\mathcal{E}_{N_H}(t)$ is the sum of the kinetic and potential energies.

\end{proof}

\subsection{Weighted $L^2$ stability of $f_{N_H}$}

We now establish the discrete counterpart of Proposition \ref{prop:L2cont}, namely the stability in an appropriately weighted $L^2$ norm. More precisely, with $\alpha$ defined by \eqref{defalpha}, we consider the weight $\omega(t,v):=\sqrt{2\pi}\,e^{(\alpha(t)v)^2/2}$, and denote by $\|\cdot\|_{\omega(t)}$ the corresponding weighted $L^2$ norm. Using the definition \eqref{fseries} of $f_{N_H}$, we have 
\begin{equation*}
\|f_{N_H}(t)\|_{\omega(t)}^2=\sum_{n,m=0}^{N_H-1}\int_0^L\int_{\RR}\alpha\,C_n\,C_m\,\Psi_m(t,v)H_n(\alpha(t)v)\dD\, v\dD x,
\end{equation*}
which gives, using orthogonality property \eqref{proportho}, the following expression for the weighted $L^2$ norm of~$f_{N_H}(t)$:
\begin{equation}
  \label{normL2cn}
\|f_{N_H}(t)\|_{\omega(t)}^2\,=\,\sum_{n=0}^{N_H-1}\alpha\,\int_0^LC_n^2\,\dD x.
\end{equation}

\begin{prop}
  \label{prop:stab_L2dv}
Assuming that $\|f_{N_H}(0)\|_{\omega(0)}<+\infty$, the distribution function $f_{N_H}$ given by the truncated series \eqref{fseries} satisfies the following stability estimate:
\begin{equation*}
\|f_{N_H}(t)\|_{\omega(t)}\,\leq\, \|f_{N_H}(0)\|_{\omega(0)}\, e^{t/4\gamma}, \qquad\forall t\geq 0,
\end{equation*}
where $\gamma>0$ is the constant appearing in the definition \eqref{defalpha} of $\alpha$.
\end{prop}

\begin{proof}
We compute the time derivative of $\|f_{N_H}(t)\|_{\omega(t)}^2$ defined by \eqref{normL2cn}:
\begin{equation*}
\frac{\dD}{\dD t}\|f_{N_H}(t)\|_{\omega(t)}^2=\sum_{n=0}^{N_H-1}\left(\alpha\int_0^L \partial_t C_n\,C_n\,\dD x+\frac{\alpha'}{2}\int_0^L C_n^2\,\dD x\right).
\end{equation*}
Thanks to equation \eqref{fvcn} with $\varphi_n=\alpha \, C_n$, we then have to estimate
\begin{align}
\frac{1}{2}\frac{\dD}{\dD t}\sum_{n=0}^{N_H-1}\alpha(t)\int_0^L\,C_n(t,x)^2\,\dD x =& \sum_{n=0}^{N_H-1}\int_0^L\alpha'\left[\left(n+\frac{1}{2}\right)C_n^2+\sqrt{n(n-1)}\,C_n\,C_{n-2}\right]\,\dD x \nonumber\\
&+ \sumn \int_0^L \left(\sqrt{n}\,C_{n-1}+\sqrt{n+1}\,C_{n+1}\right)\partial_x C_n\,\dD x \nonumber \\
&+ \sum_{n=0}^{N_H-1}\int_{0}^L \alpha^2\,\sqrt{n}\,E_{N_H}\,C_n\,C_{n-1}\,\dD x. \label{normL2cn1}
\end{align}
First of all, the transport term vanishes. Indeed, reindexing the sum over $n$ and using that $C_{-1}=0=C_{N_H}$,  we have
\[\sumn \int_0^L \left(\sqrt{n}\,C_{n-1}+\sqrt{n+1}\,C_{n+1}\right)\partial_x C_n\,\dD x = \sumn\int_0^L \sqrt{n}\,\partial_x(C_{n-1}\,C_n)\,\dD x = 0. \]
Then on the one hand,  using again that $C_n=0$ for all $n<0$ and $n\geq N_H$, we have
\begin{equation*}
\sum_{n=0}^{N_H-1}\left[\left(n+\frac{1}{2}\right)C_n^2+\sqrt{n(n-1)}\,C_n\,C_{n-2}\right]=\frac{1}{2}\sum_{n=1}^{N_H+1}\left(\sqrt{n}\,C_n+\sqrt{n-1}\,C_{n-2}\right)^2.
\end{equation*}
On the other hand, we notice that
\begin{equation*}
\sum_{n=0}^{N_H-1} \sqrt{n}\,C_n\,C_{n-1}=\frac{1}{2}\sum_{n=1}^{N_H+1}C_{n-1}\left(\sqrt{n}\,C_n+\sqrt{n-2}\,C_{n-2}\right).
\end{equation*}
Gathering these two identities in \eqref{normL2cn1}, we get
\begin{align*}
\frac{1}{2}\frac{\dD}{\dD t}\sum_{n=0}^{N_H-1}\alpha\int_0^L\,C_n^2\dD x\, =& \,\frac{1}{2}\sum_{n=1}^{N_H+1}\int_0^L \alpha'\left(\sqrt{n}\,C_n+\sqrt{n-1}\,C_{n-2}\right)^2 \dD x\\
&\, +\frac{1}{2}\sum_{n=1}^{N_H+1}\int_0^L E_{N_H}\,\alpha^2\,C_{n-1}\left(\sqrt{n}\,C_n+\sqrt{n-1}\,C_{n-2}\right)\dD x.
\end{align*}
Applying Young inequality in the second sum, this provides:
\begin{align*}
\frac{1}{2}\frac{\dD}{\dD t}\sum_{n=0}^{N_H-1}\alpha\int_0^L\,C_n^2\dD x\, \leq & \,\frac{1}{2}\sum_{n=1}^{N_H+1}\int_0^L \alpha'\left(\sqrt{n}\,C_n+\sqrt{n-1}\,C_{n-2}\right)^2 \dD x\\
&\, +\frac{1}{2}\sum_{n=1}^{N_H+1}\int_0^L \left(\frac{\gamma}{2} \,\|E_{N_H}\|_\infty^2\,\alpha^3\,\left(\sqrt{n}\,C_n+\sqrt{n-1}\,C_{n-2}\right)^2+\frac{1}{2\,\gamma}\,\alpha\,C_{n-1}^2\right)\,\dD x.
\end{align*}
Yet, using definition \eqref{defalpha} of $\alpha$, we have that 
\[\alpha'(t)=-\frac{\gamma}{2}\|E_{N_H}(t)\|_{\infty}^2\,\alpha(t)^3,\]
which yields
\begin{align*}
\frac{1}{2}\frac{\dD}{\dD t}\sum_{n=0}^{N_H-1}\alpha\int_0^L\,C_n^2\dD x\, \leq & \, \frac{1}{4\,\gamma}\sum_{n=1}^{N_H+1}\alpha\int_0^L C_{n-1}^2\dD x\\
\leq & \,\frac{1}{4\,\gamma}\sum_{n=0}^{N_H-1}\alpha\int_0^L C_{n}^2\dD x.
\end{align*}
Using Gr\"onwall's lemma, this gives the expected result.
\end{proof}

\subsection{Control of $\alpha$}

Using the definition \eqref{normL2cn} of $\|f_{N_H}(t)\|_{\omega(t)}$, the stability result established in Proposition \ref{prop:stab_L2dv} gives a stability result in $L^2(0,L)$ for the coefficients $C_n$, provided that $\alpha(t)$ is bounded from below by a positive constant. Due to definition \eqref{defalpha} of $\alpha(t)$, this is achieved as soon as $\|E_{N_H}(t)\|_\infty$ is controlled. This result is given in the following proposition.

\begin{prop}\label{prop:controlE}
  Under assumptions of Proposition \ref{prop:stab_L2dv}, the solution $E_{N_H}(t,\cdot)$ of \eqref{fvps} satisfies that for all $T>0$, there exists a constant $C_T>0$ depending on $T$ such that for all $t\in [0,T]$,
  $$
  \|E_{N_H}(t)\|_\infty\leq C_T.
  $$
\end{prop}

\begin{proof}
Since we are in one space dimension, by Sobolev inequality, there exists a constant $c>0$ such that for all $t\geq 0$,
\begin{equation*}
\|E_{N_H}(t)\|_{\infty}^2\leq c\,\|E_{N_H}(t)\|_{H^1}^2.
\end{equation*}
Moreover, since $\int_0^LE_{N_H}\,\dD x=0$, Poincar\'e-Wirtinger inequality applies and then there exists $c'>0$ such that
\begin{equation*}
\|E_{N_H}(t)\|_{\infty}^2\leq c'\,\|\partial_x E_{N_H}(t)\|_{2}^2.
\end{equation*}
Taking $\zeta=\partial_x E_{N_H}$ in the second equation of \eqref{fvps} and applying Cauchy-Schwarz inequality gives
\begin{equation*}
\|\partial_x E_{N_H}(t)\|_{2}^2=\int_0^L\partial_x E_{N_H}\,(C_0-\rho_0)\,\dD x=\int_0^L \partial_x E_{N_H}\,C_0\, \dD x \leq \|\partial_x E_{N_H}(t)\|_2\,\|C_0(t)\|_2,
\end{equation*}
and then one obtains
\begin{equation*}
\|E_{N_H}(t)\|_{\infty}^2\,\leq\, c'\,\|C_0(t)\|_2^2.
\end{equation*}
Hence, we need to control $\|C_0(t)\|_2$ as, 
\begin{align*}
\int_0^L C_0^2\,\dD x&=\int_0^L\left(\int_\RR f_{N_H}(t,x,v)\,\omega(t,v)^{1/2}\,\omega(t,v)^{-1/2}\,\dD v\right)^2\dD x\\
&\leq \int_0^L \left(\int_{\RR}f_{N_H}(t,x,v)^2\,\omega(t,v)\,\dD v\right)\left(\int_{\RR}\omega(t,v)^{-1}\,\dD v\right)\,\dD x\\
&\leq \frac{1}{\alpha(t)}\,\|f_{N_H}(t)\|_{\omega(t)}^2,
\end{align*}
which depends on $\alpha$, that is, on $\|E\|_\infty$. Then, using the definition \eqref{defalpha} of $\alpha$, it yields
\begin{equation*}
\frac{\alpha_0}{\alpha(t)}=\left(1+\gamma\int_0^t\|E_{N_H}(s)\|_\infty^2\,\dD s\right)^{\frac{1}{2}}\leq 1+\gamma\int_0^t\|E_{N_H}(s)\|_\infty^2\,\dD s,
\end{equation*}
which gives
\begin{equation*}
\|E_{N_H}(t)\|_\infty^2\leq c'\,\|f_{N_H}(t)\|_{\omega(t)}^2\left(1+\gamma\int_0^t\|E_{N_H}(s)\|_\infty^2\,\dD s\right).
\end{equation*}
Furthermore, by Proposition \ref{prop:stab_L2dv}, we have for all $t\in [0,T]$,
\begin{equation*}
\|f_{N_H}(t)\|_{\omega(t)}^2\leq c_T:=\|f_{N_H}(0)\|_{\omega(0)}^2\,e^{T/2\gamma}.
\end{equation*}
Thus, 
\begin{equation*}
\|E_{N_H}(t)\|_\infty^2\leq c'\,c_T\left(1+\gamma\int_0^t\|E_{N_H}(s)\|_\infty^2\,\dD s\right).
\end{equation*}
By Gr\"onwall's lemma, we conclude  that for all $t\in [0,T]$,
\begin{equation*}
\|E_{N_H}(t)\|_\infty^2\leq c'\,c_T\,e^{c'c_T \gamma \,t}.
\end{equation*}
\end{proof}

\subsection{Filtering technique} As proposed in \cite{Filbet2020}, we finally apply a filtering procedure to reduce the effects of the Gibbs phenomenon inherent to spectral methods \cite{Hesthaven2007}. The filter consists in multiplying some spectral coefficients $C_n$ in \eqref{fseries} by a scaling factor $\sigma$ to reduce the amplitude of high frequencies, for any $N_H\geq 4$. Concretely, coefficients $C_n$ in \eqref{fseries} are replaced by $\tilde{C}_n$, with
\begin{equation*}
\tilde{C}_n=C_n\,\sigma\left(n/N_H\right).
\end{equation*}
As in \cite{Filbet2020}, we consider the Hou-Li's filter, which reads
 $$
 \sigma(s)\,=\,
 \left\{
   \begin{array}{ll}
     1\,, & \textrm{if\,\,} 0\leq |s|\leq 2/3\,,
     \\[1.1em]
     e^{-\beta\,|s|^\beta}\,,&\textrm{if \,\,} |s|>2/3\,,
     \end{array}
   \right.
 $$
 where the coefficient $\beta$ is chosen as $\beta = 36$. 
 
 Remark that since the filter is applied only when $N_H\geq 4$, the filtering process does not modify the three first coefficients $C_0,\, C_1$ and $C_2$. Then conservations of mass, momentum and total energy established in Proposition \ref{prop:1} are not affected. Moreover, since $\sigma\leq 1$, this process does not impact the stability result established in propositions \ref{prop:stab_L2dv} and \ref{prop:controlE} either.

%
\section{A class of spatial discretizations}
\label{sec:3}
\setcounter{equation}{0}
\setcounter{figure}{0}
\setcounter{table}{0}

This section is devoted to the space discretization of the
 Vlasov-Poisson system written in the form \eqref{cn}-\eqref{ps}. We first consider a spectral method, and then a class of discontinuous Galerkin methods.


\subsection{Fourier method}\label{sec:Fourier}
 
When considering periodic boundary conditions, spectral method with Fourier basis is a natural choice for the spatial discretization. As for example in \cite{Kormann2021}, we consider an odd number $2N_x+1$ of Fourier modes, and look for an approximation $f_{N_H,N_x}$ of the distribution function $f$ solution of Vlasov equation \eqref{vlasov}  defined as
\begin{equation}\label{fFourierHermite}
f_{N_H,N_x}(t,x,v)=\sumn\sumj c_n^j(t)\,e_j(x)\,\Psi_n(t,v),
\end{equation}
where
\[e_j(x):=\frac{1}{\sqrt{L}}\,e^{2\text{i}\pi j x/L}.\]
This corresponds to approximate the Hermite modes $C_n(t,x)$ defined in Section \ref{sec:hermite} as
\begin{equation*}
 C_{n,N_x}(t,x)=\sumj c_n^j(t)\,e_j(x).
 \end{equation*}
The electric field $E_{N_H}$ is also approximated in the Fourier basis as
\begin{equation}\label{EFourierHermite}
E_{N_H,N_x}(t,x)=\sumj E^j(t)\,e_j(x).
\end{equation}
Equations satisfied by the coefficients $(c_n^j)$ and $E^j$ are obtained by considering the weak formulation \eqref{fvcn}--\eqref{fvps} for $C_{n,N_x}$ and $E_{N_H,N_x}$ in the space
\[V_{N_x}:= \text{span}\left\{e_j,\,j=-N_x,\ldots,N_x\right\},\]
and taking $e_j$ as test functions.
Remark that here the function $\alpha$ arising in the definition \eqref{hbasisf} is given~by
\begin{equation*}
\alpha(t)\,\,=\,\,\alpha_0\,\left(1+\gamma\,\int_0^t\|E_{N_H,N_x}(s)\|_\infty^2 \dD s\right)^{-1/2}.
\end{equation*}

For this Fourier-Hermite discretization of the Vlasov-Poisson equation, all the results presented in the previous section are satisfied. They are summarized in the following theorem and can be proved exactly in the same way as propositions \ref{prop:1}, \ref{prop:stab_L2dv} and \ref{prop:controlE} by considering $V=V_{N_x}$ in the weak formulation \eqref{fvcn}--\eqref{fvps}.

\begin{thm}
  For any $N_H,\,N_x \in \NN$, consider the approximate solution $(f_{N_H,N_x},E_{N_H,N_x})$ given by \eqref{fFourierHermite}-\eqref{EFourierHermite}, where the coefficients $(c_n^j,E^j)_j$ satisfy \eqref{fvcn}--\eqref{fvps} in $V_{N_x}$. Then,
  \begin{itemize}
    \item mass, momentum and total energy are conserved, that is
\[ \frac{\dD}{\dD t}\left[\frac{c_k^0}{\alpha^k}\right](t)\,=\,0,\quad k=0,\,1,  \]
and for the total energy,
\[\mathcal{E}_{N_H,N_x}(t):=\frac{1}{2}\frac{1}{\alpha^2(t)}\left(\sqrt{2}\,c_2^0(t)+c_0^0(t)\right)+\frac{1}{2}\sumj |E^j(t)|^2=\mathcal{E}_{N_H,N_x}(0).\]
\item Furthermore, assuming that $\|f_{N_H,N_x}(0)\|_{\omega(0)}<+\infty$, the approximate distribution function $f_{N_H,N_x}$ satisfies the following stability estimate:
  \[ \|f_{N_H,N_x}(t)\|_{\omega(t)}\,\leq \, \|f_{N_H,N_x}(0)\|_{\omega(0)}\,e^{t/4\gamma}.
  \]
\item Finally, there exists a constant $C_T>0$ such that for all $t\in [0,T]$,
  $$
  \|E_{N_H,N_x}(t)\|_{\infty}\leq C_T.
  $$
  \end{itemize}
\end{thm} 
 

\subsection{Discontinuous Galerkin method for the Vlasov equation} 

In the spirit of \cite{Ayuso2011,Filbet2020}, we now consider a discontinuous Galerkin approximation of the Vlasov equation \eqref{cn}. This method guarantees conservation of discrete mass as well as weighted $L^2$ stability. Then, to compute an approximation of the Poisson equation, several approaches arise:
\begin{itemize}
\item we will consider a local discontinuous Galerkin method for the Poisson equation as in \cite{Filbet2020}. For this choice, we prove conservation of total energy. However, we do not have control the $L^\infty$ norm of the electric field in this case.
\item  we will also look for a mixed finite element approximation for the Poisson equation. We establish that the obtained approximate solution satisfies semidiscrete counterpart of Proposition \ref{prop:controlE}. However, we cannot obtain conservation of total energy with this choice of approximation.
\end{itemize}

Let us now define the discontinuous Galerkin (DG) scheme for the Vlasov equation with Hermite spectral basis in velocity \eqref{cn}. We first introduce some notations and start with
$\{\xR\}_{i=0}^{i=N_x}$, a partition of $[0,L]$, with $x_{\frac12}=0$,
$x_{N_x+\frac12}=L$. Each element is denoted as $I_i=[\xL,
\xR]$ with its length $h_i$, and $h=\max_i h_i$.

Given any $k\in\mathbb{N}$, we define a finite dimensional discrete piecewise polynomial space
\begin{equation}\label{defVhk}
V_h^k=\left\{u\in L^2(0,L): u|_{I_i}\in P_k(I_i), \forall i\right\}\,,
\end{equation}
where the local space $P_k(I)$ consists of polynomials of degree at
most $k$ on the interval $I$. We further denote the jump $[u]_\iR$
and the average $\{u\}_\iR$ of $u$ at $x_\iR$ defined as
$$
[u]_\iR\,=\,{u(x_\iR^+)\,-\,u(x_\iR^-)} \quad{\rm and}\quad
\{u\}_\iR\,=\,\frac12\,\left(u(x_\iR^+)\,+\,u(x_\iR^-)\right)\,, \quad \forall\, i\,,
$$
where $u(x^\pm)=\lim_{\Delta x\rightarrow 0^\pm} u(x+\Delta x)$.  We
also denote

\[  u_\iR=u(x_\iR)\,, \qquad u^\pm_\iR=u(x^\pm_\iR)\,. \]

From these notations, we apply a semi-discrete discontinuous Galerkin
method for \eqref{cn} as follows. We look for an approximation $C_{n,h}(t,\cdot) \in V_h^k$, such that for any $\varphi_n \in V_h^k$, we have
\begin{align}
  \frac{\dD}{\dD t}\int_{I_j}C_{n,h}\,\varphi_n\,\dD x&-\frac{\alpha'}{\alpha}\int_{I_j}\left(n\,C_{n,h}+\sqrt{(n-1)n}\,C_{n-2,h}\right)\varphi_n\,\dD x+a_n^j(g_n,\varphi_n)\nonumber
  \\
&-\alpha\sqrt{n}\int_{I_j}E_{N_H,h}\,C_{n-1,h}\,\varphi_n\,\dD x=0,\quad 0\leq n \leq N_H-1, \label{dgcn}
\end{align}
where $a^j_n$ is defined by 
\beq
\label{anh}
\left\{
  \begin{array}{l}
\ds a^j_{n}(g_{n},\varphi_n) \,=\, -\int_{I_j}g_{n}\,\testR'_n \,
    \dD
    x\,+\,\hat{g}_{n,j+\f12}\,\testR^-_{n,j+\f12}-\hat{g}_{n,j-\f12}\,\testR^+_{n,j-\f12}\,,
    \\[1.1em]
\ds g_{n} \,=\,\frac{1}{\alpha}\left(\sqrt{n+1}\,C_{n+1,h}\,+\,\sqrt{n}\,C_{n-1,h}\right)\,.
\end{array}\right.
\eeq

The numerical flux $\hat{g}_{n}$ in \eqref{anh} is given by
\beq
\label{lfflx}
\hat{g}_{n}\,=\,\f12\left[g^-_{n}+g^+_{n}-\delta_n\,\left(C^+_{n,h}\,-\,C^-_{n,h}\right)\right]\,,
\eeq
with the numerical viscosity coefficient defined as $\delta_0=0$, corresponding to a centered flux, and for $1\leq n\leq N_H-1$, we consider the global Lax-Friedrichs flux with $\delta_n=\delta=\sqrt{N_H}/\alpha$. The choice of the centered flux in the case $n=0$ is made to recover the conservation of the semi-discrete total energy (see Proposition \ref{prop:conservationsDG}). Once again, due to the definition of $\delta_n$, it appears crucial to have a positive lower bound for $\alpha$. This property is discussed in what follows, according to the discretization chosen for the Poisson equation.

The approximate solution of \eqref{vlasov} obtained using Hermite polynomials in velocity variable and discontinuous Galerkin discretization in space is then defined by 
\beq
\label{dgfseries}
f_{N_H,N_x}(t,x,v)=\sum_{n=0}^{N_H-1}C_{n,h}(t,x)\Psi_n(t,v)\,,
\eeq
where $C_{n,h}$ satisfy \eqref{dgcn} and $\Psi_n$ are the basis functions defined by \eqref{hbasisf}. The weighted $L^2$ norm of $f_{N_H,N_x}$ is then given by
\begin{equation}\label{dgL2norm}
\|f_{N_H,N_x}(t)\|_{\omega(t)}^2\,=\,\sum_{n=0}^{N_H-1}\alpha(t)\int_{0}^LC^2_{n,h}(t,x)\,\dD x.
\end{equation}
Throughout this section, the function $\alpha$ is defined by
\begin{equation}\label{defalphaDG1}
\alpha(t)\,\,=\,\,\alpha_0\,\left(1+\gamma\,\int_0^t\|E_{N_H,h}(s)\|_\infty^2 \dD s\right)^{-1/2},
\end{equation}
where $E_{N_H,h}$ is the chosen spatial approximation of the electric field $E_{N_H}$. 

Similarly as in the continuous case, we establish the following stability result.

\begin{prop}
  \label{prop:stab_L2dvdx}
  Consider the semi-discrete approximate solution $f_{N_H,N_x}$ given by the truncated series \eqref{dgcn}-\eqref{dgfseries}. Then, we have
\begin{itemize}
\item Conservation of mass
  \[ \frac{\dD}{\dD t}\int_0^LC_{0,h}\dD x \,=\,0.
  \]

\item Furthermore,  assuming that $\|f_{N_H,N_x}(0)\|_{\omega(0)}<+\infty$, we have
\begin{equation}
\label{eq1:stabL2dvdx}
\frac{1}{2}\,\frac{\dD}{\dD t}\|f_{N_H,N_x}(t)\|_{\omega(t)}^2+\alpha(t)\sum_{n=1}^{N_H-1}\sum_{j}\,\delta_n\,[C_{n,h}]_{j-\frac{1}{2}}^2\leq \frac{1}{2\,\gamma}\,\|f_{N_H,N_x}(t)\|_{\omega(t)}^2, \quad\forall t\geq 0,
\end{equation}
from which we deduce
\begin{equation}\label{eq2:stabL2dvdx}
\|f_{N_H,N_x}(t)\|_{\omega(t)}\leq \|f_{N_H,N_x}(0)\|_{\omega(0)}\,e^{t/4\gamma},\quad \forall t\geq 0.
\end{equation}
\end{itemize}
\end{prop}

\begin{rem}\label{rem:momentum}
  Note that compared to the results obtained in propositions \ref{prop:L2cont} and \ref{prop:stab_L2dv}, an additional dissipation term appears in \eqref{eq1:stabL2dvdx}, arising from the discontinuous Galerkin discretization.
\end{rem}

\begin{proof}
Throughout this proof, without any confusion, we will drop the indices $N_H$ and $h$ to lighten the notations.

Let us underline that the discontinuous Galerkin method naturally preserves mass. Indeed, since the equation \eqref{dgcn} for $n=0$ only contains a convective term, we get that
\[
\frac{\dD}{\dD t}\int_0^L C_{0}\,\dD x=0.
\]
Then, using \eqref{dgL2norm}, we have
\begin{equation*}
\frac{1}{2}\frac{\dD}{\dD t}\|f_{N_H,N_x}(t)\|_{\omega(t)}^2=\sum_{n=0}^{N_H-1}\int_0^L\left(\alpha\,C_n\,\partial_t C_n+\frac{\alpha'}{2}\,C_n^2\right)\dD x.
\end{equation*}
Taking $\varphi_n=C_n$ as test function in \eqref{dgcn}, we get
\begin{align}
\frac{1}{2}\frac{\dD}{\dD t}\|f_{N_H,N_x}(t)\|_{\omega(t)}^2=& \sum_{n=0}^{N_H-1}\alpha'\int_{0}^L\left(\left(n+\frac{1}{2}\right)\,C_n+\sqrt{(n-1)n}\,C_{n-2}\right)C_n\dD x \label{dtfnhnx}\\[0.9em]
-&\sum_{n=0}^{N_H-1}\sum_j \alpha\,a_n^j(g_n,C_n) +\sum_{n=0}^{N_H-1}\alpha^2\sqrt{n}\int_0^L E\,C_{n-1}\,C_n\,\dD x. \nonumber
\end{align}
Let us first consider the transport term. Indeed, other terms will be treated exactly in the same way as in the proof of Proposition \ref{prop:stab_L2dv}. Using the definition \eqref{anh} of $a_n^j$, we have
\begin{equation*}
\sum_{n=0}^{N_H-1}\sum_j \alpha\,a_n^j(g_n,C_n)=T_1+T_2,
\end{equation*}
with
$$
\left\{
\begin{array}{ll}
T_1&\ds = -\sum_{n=0}^{N_H-1}\sum_j\alpha\int_{I_j}g_n\,\partial_x C_n\,\dD x,\\
T_2&\ds = \sumn\sum_j\alpha\left(\hat{g}_{n,j+\frac{1}{2}}C_{n,j+\frac{1}{2}}^--\hat{g}_{n,j-\frac{1}{2}}C_{n,j-\frac{1}{2}}^+\right).
\end{array}\right.
$$
We now consider the first term $T_1$. By definition \eqref{anh} of $g_n$, we get
\begin{equation*}
T_1=-\sumn\sum_j\int_{I_j}\left(\sqrt{n+1}\,C_{n+1}+\sqrt{n}\,C_{n-1}\right)\partial_xC_n\,\dD x.
\end{equation*}
Now, reindexing the sum over $n$ and using that $C_{-1}=0=C_{N_H}$, we obtain
\begin{align*}
T_1&=-\sumn\sum_j\sqrt{n}\,\int_{I_j}\partial_x(C_{n-1}C_n)\,\dD x\\
&=-\sumn\sum_j\sqrt{n}\left((C_{n-1}C_n)_{j+\frac{1}{2}}^--(C_{n-1}C_n)_{j-\frac{1}{2}}^+\right).
\end{align*}
Using the periodic boundary conditions, it finally yields
\begin{equation*}
T_1=\sumn\sum_j\sqrt{n}\,\left[C_{n-1}C_n\right]_{j-\frac{1}{2}}.
\end{equation*}
We now deal with the second term $T_2$. By definition \eqref{anh} of $g_n$, we compute
\[g_{n,j-\frac{1}{2}}^-+g_{n,j-\frac{1}{2}}^+=\frac{2}{\alpha}\left(\sqrt{n+1}\,\{C_{n+1}\}+\sqrt{n}\,\{C_{n-1}\}\right)_{j-\frac{1}{2}},\]
which yields
\begin{equation*}
\hat{g}_{n,j-\frac{1}{2}}=\frac{1}{\alpha}\left(\sqrt{n+1}\,\{C_{n+1}\}+\sqrt{n}\,\{C_{n-1}\}\right)_{j-\frac{1}{2}}-\frac{\delta_n}{2}\,[C_n]_{j-\frac{1}{2}}.
\end{equation*}
Once again, taking advantage of periodic boundary conditions, we have
\begin{align*}
T_2&=-\sumn\sum_j \alpha\, \hat{g}_{n,j-\frac{1}{2}}\,[C_{n}]_{j-\frac{1}{2}}\\
&= -\sumn\sum_j  \left(\left(\sqrt{n+1}\,\{C_{n+1}\}+\sqrt{n}\,\{C_{n-1}\}\right)_{j-\frac{1}{2}}\,[C_{n}]_{j-\frac{1}{2}}-\frac{\alpha\,\delta_n}{2}\,[C_n]_{j-\frac{1}{2}}^2\right).
\end{align*}
By rearranging the terms of the sum over $n$, using once again that $C_{-1}=0=C_{N_H}$, we obtain
\begin{equation*}
T_2=-\sumn\sum_j \sqrt{n}\,\left(\{C_n\}[C_{n-1}]+\{C_{n-1}\}[C_n]\right)_{j-\frac{1}{2}}+\frac{\alpha}{2}\sum_{n=1}^{N_H-1}\sum_j \delta_n\,[C_n]_{j-\frac{1}{2}}^2.
\end{equation*}
Since 
\[\{C_n\}[C_{n-1}]+\{C_{n-1}\}[C_n]=[C_{n-1}\,C_n],\]
we finally have that
\begin{equation*}
\sumn\sum_j \alpha\,a_n^j(g_n,C_n)=T_1+T_2= \frac{\alpha}{2}\sum_{n=1}^{N_H-1}\sum_j \delta_n\,[C_n]_{j-\frac{1}{2}}^2.
\end{equation*}
Inserting this equality in \eqref{dtfnhnx}, we get
\begin{align*}
\frac{1}{2}\frac{\dD}{\dD t}\|f_{N_H,N_x}(t)\|_{\omega(t)}^2 +\frac{\alpha}{2}\sum_{n=1}^{N_H-1}\sum_j \delta_n\,[C_n]_{j-\frac{1}{2}}^2=& \sum_{n=0}^{N_H-1}\alpha'\int_{0}^L\left(\left(n+\frac{1}{2}\right)\,C_n+\sqrt{(n-1)n}\,C_{n-2}\right)C_n\dD x \\[0.9em]
 +&\sum_{n=0}^{N_H-1}\alpha^2\sqrt{n}\int_0^L E\,C_{n-1}\,C_n\,\dD x. 
\end{align*}
Finally, the right hand side can be treated exactly in the same way as in the proof of Proposition \ref{prop:stab_L2dv}, leading to \eqref{eq1:stabL2dvdx}. Nonnegativity of $\alpha(t)$ and direct application of the Gr\"onwall's lemma provides~\eqref{eq2:stabL2dvdx}.
\end{proof}

We next deal with the approximation of the electric field $E_{N_H}$. This step is crucial to get energy conservation and to control the scaling parameter $\alpha$.

To this end, we need to consider the potential function $\Phi_{N_H}(t,x)$, such that
\begin{equation}\label{approxPoisson}
\left\{
\begin{array}{l}
\ds E_{N_H}\,=\,-\f{\partial \Phi_{N_H}}{\px}\,, \\[0.9em]
\ds\f{\partial E_{N_H}}{\px} \,=\,C_{0} - \rho_0\,.
\end{array}\right.
\end{equation}
Hence we get the one dimensional Poisson equation 
$$
-\f{\partial^2 \Phi_{N_H}}{\px^2} \,=\,C_{0} \,-\, \rho_0\,.
$$
We now propose two methods to discretize this problem in space and discuss their properties.

\subsubsection{Discontinuous Galerkin approximation of the Poisson equation} \label{sec:DGPoisson}

As in \cite{Filbet2020}, we now consider the DG approximation \eqref{dgcn}--\eqref{lfflx} for the Vlasov equation together with a local discontinuous Galerkin method for the Poisson equation \eqref{ps}. We prove conservation of the discrete total energy. However, in this framework, we are not able to control the $L^\infty$ norm of the electric field as in Proposition \ref{prop:controlE}, and conservation of momentum is not achieved (see Remark \ref{rem:momentum}).  
 
The Vlasov equation being approximated by \eqref{dgcn}, we now define a DG approximation of the Poisson equation \eqref{approxPoisson}. We
look for a couple $\Phi_{N_H,h}(t,\cdot), E_{N_H,h}(t,\cdot) \in V_h^k$, such
that for any $\testP$ and $\testE \in V_h^k$, we have 
\beq
\label{dgP}
\left\{
\begin{array}{l}
\ds+\int_{I_j}\Phi_{N_H,h}\,\testP'\,\dD x \, -\,
  \hat{\Phi}_{N_H,h,j+\f12}\,\testP^-_{j+\f12} \,+\,\hat{\Phi}_{N_H,h,j-\f12}\,\testP^+_{j-\f12}\,=\,\int_{I_j}
  E_{N_H,h}\,\testP\,\dD x\,,
  \\[1.1em]
\ds -\int_{I_j}E_{N_H,h}\,\testE'\,\dD x \,+\,
  \hat{E}_{N_H,h,j+\f12}\,\testE^-_{j+\f12}
  \,-\,\hat{E}_{N_H,h,j-\f12}\,\testE^+_{j-\f12}\,=\,\int_{I_j} \left(
  C_{0,h}\,-\,\rho_0\right)\,\testE\,\dD x\,,
\end{array}\right.
\eeq
where the numerical fluxes $\hat{\Phi}_{N_H,h}$ and $\hat{E}_{N_H,h}$ in
\eqref{dgP}  are taken as 
\beq
\label{fluxps}
\left\{
\begin{array}{l}
  \hat{\Phi}_{N_H,h} \,=\, \{\Phi_{N_H,h}\}\,, \\[0.9em]
  \hat{E}_{N_H,h} \,=\, \{E_{N_H,h}\}\,-\,\beta\,[\Phi_{N_H,h}]\,,
\end{array}\right.
  \eeq
with $\beta$ either being a positive constant or a constant multiplying by $1/h$.

We now focus on the conservation properties of this DG method. As already mentioned, the conservation of mass is guaranteed, whatever the choice of discretization for the Poisson equation. Moreover, conservation of the total energy is achieved thanks to the particular choice of the (centered) numerical flux $\hat{g}_0$ in the equation \eqref{dgcn} for $n=0$. Unfortunately, the momentum is not conserved due to the contribution of the source terms in \eqref{dgcn}. In \cite{Filbet2020}, a slight modification of the flux for the unknown $C_{1,h}$ was proposed to ensure the momentum conservation. However, if we apply this modification here, we lose the $L^2$ stability established in Proposition \ref{prop:stab_L2dvdx}. 

\begin{thm}
  \label{prop:conservationsDG}
  For any $N_H\geq 3$, we consider the solution $(C_{n,h},E_{N_H,h},\Phi_{N_H,h})$ to the system \eqref{dgcn}--\eqref{lfflx} together with \eqref{dgP}--\eqref{fluxps}. Then we have
  \begin{itemize}
  \item Conservation of mass.
  \item Furthermore,  assuming that $\|f_{N_H,N_x}(0)\|_{\omega(0)}<+\infty$,  we get
    $$
\|f_{N_H,N_x}(t)\|_{\omega(t)}\leq \|f_{N_H,N_x}(0)\|_{\omega(0)}\,e^{t/4\gamma},\quad \forall t\geq 0.
    $$
  \item The discrete total energy, defined as
\begin{equation}\label{defEtotdg}
\mathcal{E}_{N_H,h}(t)=\frac{1}{2}\int_0^L\left(\frac{1}{\alpha^2}\,(\sqrt{2}\,C_{2,h}+C_{0,h})+|E_{N_H,h}|^2\right)\dD x+\frac{1}{2}\,\beta\sum_j \,[\Phi_{N_H,h}]_{j-\frac{1}{2}}^2,
\end{equation}
is preserved with respect to time.
\end{itemize}
  \end{thm}
\begin{proof}
First and second items follows from Proposition \ref{prop:stab_L2dvdx}.  Then, once again, we omit here the indices $N_H$ and $h$ and first compute the time derivative of the kinetic energy, using mass conservation:
\begin{equation*}
\frac{1}{2}\frac{\dD}{\dD t}\left(\frac{1}{\alpha^2}\int_0^L(\sqrt{2}\,C_2+C_0)\dD x\right) = -\frac{\alpha'}{\alpha^3}\int_0^L(\sqrt{2}\,C_2+C_0)\,\dD x+\frac{1}{\sqrt{2}\,\alpha^2}\int_0^L\partial_t C_2\,\dD x.
\end{equation*}
Using $\varphi_2=\frac{1}{\sqrt{2}\,\alpha^2}$ in \eqref{dgcn} for $n=2$ and summing over $j$, we have
\[\frac{1}{\sqrt{2}\,\alpha^2}\int_0^L\partial_t C_2\,\dD x-\frac{\alpha'}{\alpha^3}\int_0^L(\sqrt{2}\,C_2+C_0)\,\dD x=\frac{1}{\alpha}\int_0^L E\,C_1\,\dD x,\]
which yields that
\begin{equation}\label{dgkinenergy1}
\frac{1}{2}\frac{\dD}{\dD t}\left(\frac{1}{\alpha^2}\int_0^L(\sqrt{2}\,C_2+C_0)\dD x\right) =\frac{1}{\alpha}\int_0^L E\,C_1\,\dD x.
\end{equation}
On the one hand, to compute the right hand side of \eqref{dgkinenergy1}, we choose $\eta=\frac{C_1}{\alpha}$ in the first equation of~\eqref{dgP}:
\begin{equation}\label{dgkinenergy2}
\frac{1}{\alpha}\int_{I_j}E\,C_1\,\dD x=\frac{1}{\alpha}\int_{I_j}\Phi\,\partial_x C_1\,\dD x-\frac{1}{\alpha}\,\hat{\Phi}_{j+\frac{1}{2}}\,C_{1,j+\frac{1}{2}}^-+\frac{1}{\alpha}\,\hat{\Phi}_{j-\frac{1}{2}}\,C_{1,j-\frac{1}{2}}^+.
\end{equation}
On the other hand, we take $\varphi_0=\Phi$ in \eqref{dgcn} for $n=0$ and use definition \eqref{anh} to get
\begin{align}
\int_{I_j}\partial_tC_0\,\Phi\,\dD x &=-  a_0^j(g_0,\Phi)=-\frac{1}{\alpha}\,a_0^j(C_1,\Phi)\nonumber\\
&=\frac{1}{\alpha}\int_{I_j} C_1\,\partial_x\Phi\,\dD x-\frac{1}{\alpha}\,\hat{C}_{1,j+\frac{1}{2}}\Phi_{j+\frac{1}{2}}^-+\frac{1}{\alpha}\,\hat{C}_{1,j-\frac{1}{2}}\,\Phi_{j-\frac{1}{2}}^+.\label{dgkinenergy3}
\end{align}
Adding \eqref{dgkinenergy2} and \eqref{dgkinenergy3} and summing over $j$, we obtain
\begin{align*}
\frac{1}{\alpha}\int_0^L E\,C_{1}\,\dD x+\int_0^L\partial_tC_0\,\Phi\,\dD x&=\frac{1}{\alpha}\sum_j\int_{I_j}\partial_x(C_1\,\Phi)\,\dD x+\frac{1}{\alpha}\sum_j\left(\hat{\Phi}[C_1]+\hat{C}_1[\Phi]\right)_{j-\frac{1}{2}}\\
&=\frac{1}{\alpha}\sum_j\left(-[C_1\,\Phi]+\hat{\Phi}[C_1]+\hat{C}_1[\Phi]\right)_{j-\frac{1}{2}}.
\end{align*}
Since $[C_1\,\Phi]=[C_1]\{\Phi\}+\{C_1\}[\Phi]$ and $\{\Phi\}=\hat{\Phi}$ by \eqref{fluxps}, we obtain
\begin{equation}\label{dgkinenergy3bis}
\frac{1}{\alpha}\int_0^L E\,C_1\,\dD x+\int_0^L \partial_t C_0\,\Phi\,\dD x=\frac{1}{\alpha}\sum_j \left( \hat{C}_1-\{C_1\}\right)[\Phi]_{j-\frac{1}{2}}.
\end{equation}
By definition \eqref{lfflx} of $\hat{g}_0$ with $\delta_0=0$, we deduce that $\hat{C}_1=\{C_1\}$, which cancels the right hand side of~\eqref{dgkinenergy3bis}. We now have to treat the second term of the left hand side.  To this end, let us take the time derivative of the second equation in \eqref{dgP} and choose $\zeta=\Phi$ as test function. It gives 
\begin{equation}\label{dgkinenergy4}
\int_{I_j}\partial_tC_0\,\Phi\,\dD x=-\int_{I_j}\partial_t E\,\partial_x \Phi\,\dD x+\widehat{\partial_t E}_{j+\frac{1}{2}}\,\Phi_{j+\frac{1}{2}}^- -\widehat{\partial_t E}_{j-\frac{1}{2}}\,\Phi_{j-\frac{1}{2}}^+.
\end{equation}
Taking $\eta=-\partial_t E$ in the first equation of \eqref{dgP}, we obtain
\begin{equation*}
-\int_{I_j}E\,\partial_t E\,\dD x=-\int_{I_j}\Phi\,\partial_x\partial_t E\,\dD x+\hat{\Phi}_{j+\frac{1}{2}}(\partial_t E)_{j+\frac{1}{2}}^- -\hat{\Phi}_{j-\frac{1}{2}}(\partial_t E)_{j-\frac{1}{2}}^+.
\end{equation*}
Performing an integration by parts of the first term of the right hand side, it provides
\begin{equation*}
-\frac{1}{2}\frac{\dD}{\dD t}\int_{I_j}E^2\,\dD x=\int_{I_j}\partial_x\Phi\,\partial_t E\,\dD x-(\Phi\,\partial_t E)_{j+\frac{1}{2}}^-+(\Phi\,\partial_t E)_{j-\frac{1}{2}}^+ +\hat{\Phi}_{j+\frac{1}{2}}(\partial_t E)_{j+\frac{1}{2}}^- -\hat{\Phi}_{j-\frac{1}{2}}(\partial_t E)_{j-\frac{1}{2}}^+.
\end{equation*}
Adding this latter equality and \eqref{dgkinenergy4}, and summing over $j$, we get
\begin{equation*}
\int_0^L\partial_t C_0\,\Phi\,\dD x-\frac{1}{2}\frac{\dD}{\dD t}\int_{0}^L E^2\,\dD x=\sum_j\left([\Phi\,\partial_t E]-\widehat{\partial_t E}\,[\Phi]-\hat{\Phi}\,[\partial_t E]\right)_{j-\frac{1}{2}}.
\end{equation*}
Now, by definition \eqref{fluxps} of the numerical fluxes, we have
\begin{align*}
[\Phi\,\partial_t E]-\widehat{\partial_t E}\,[\Phi]-\hat{\Phi}\,[\partial_t E]&=\left(\{\partial_t E\}-\widehat{\partial_t E}\right)[\Phi]+\left(\{\Phi\}-\hat{\Phi}\right)[\partial_t E]\\
&= \beta[\partial_t \Phi][\Phi]\\
&=\frac{\beta}{2}\frac{\dD}{\dD t}[\Phi]^2,
\end{align*}
which gives
\begin{equation*}
\int_0^L \partial_t C_0\,\Phi\,\dD x=\frac{1}{2}\frac{\dD}{\dD t}\left(\int_0^L E^2\,\dD x+\beta\sum_j\,[\Phi]^2_{j-\frac{1}{2}}\right).
\end{equation*}
Using this last identity together with \eqref{dgkinenergy1} and \eqref{dgkinenergy3bis}, we finally obtain:
\begin{equation}\label{conservationenergy}
\frac{\dD}{\dD t}\mathcal{E}(t)=0.
\end{equation}
\end{proof}

\begin{rem}
Concerning momentum, we have
\begin{equation}\label{dgmomentum1}
\frac{\dD}{\dD t}\int_0^L \frac{C_1}{\alpha}\,\dD x =\int_0^L\left(\frac{1}{\alpha}\,\partial_t C_1-\frac{\alpha'}{\alpha^2}\,C_1  \right)\,\dD x.
\end{equation}
Taking $\varphi_1=\frac{1}{\alpha}$ as test function in \eqref{dgcn} for $n=1$ and summing over $j$, we get
\begin{equation}\label{dgmomentum1bis}
\frac{1}{\alpha}\int_{0}^L\partial_t C_1\, \dD x -\frac{\alpha'}{\alpha^2}\int_{0}^L C_1\,\dD x=\int_0^L E\,C_0\,\dD x.
\end{equation}
Now, taking $\zeta = E$ in the second equation of \eqref{dgP}, we obtain
\begin{equation}\label{dgmomentum2}
\int_{I_j}C_0\,E\,\dD x = \rho_0\int_{I_j}E\,\dD x -\int_{I_j}E\,\partial_x E\,\dD x+\hat{E}_{j+\frac{1}{2}}E^-_{j+\frac{1}{2}}-\hat{E}_{j-\frac{1}{2}}E^+_{j-\frac{1}{2}}.
\end{equation}
To compute the first term of the right hand side, we take $\eta=1$ in the first equation of \eqref{dgP}:
\begin{equation*}
\int_{I_j}E\,\dD x= -\hat{\Phi}_{j+\frac{1}{2}}+\hat{\Phi}_{j-\frac{1}{2}}.
\end{equation*}
Using this in \eqref{dgmomentum2} and summing over $j$, it yields
\begin{align*}
\int_{0}^LC_0\, E\,\dD x &= -\frac{1}{2}\sum_j\int_{I_j}\partial_x(E^2)\dD x-\sum_j \hat{E}_{j-\frac{1}{2}}[E]_{j-\frac{1}{2}}\\ 
&= \sum_j\left(\frac{1}{2}\,[E^2]-\hat{E}[E]\right)_{j-\frac{1}{2}}.
\end{align*}
Since $[E^2]=2\{E\}[E]$, we have thanks to the definition \eqref{fluxps} of $\hat{E}$ that
\begin{equation}\label{dgmomentum3}
\int_{0}^LC_0\, E\,\dD x = \sum_j\left(\{E\}-\hat{E}\right)[E]_{j-\frac{1}{2}}=\beta\sum_j \,[\Phi]_{j-\frac{1}{2}}\,[E]_{j-\frac{1}{2}}.
\end{equation}
Putting \eqref{dgmomentum1bis} and \eqref{dgmomentum3} together in \eqref{dgmomentum1}, we finally obtain
\begin{equation}\label{conservationmomentum}
\frac{\dD}{\dD t}\int_{0}^L \frac{C_1}{\alpha}\,\dD x = \beta\sum_j \,[\Phi]_{j-\frac{1}{2}}\,[E]_{j-\frac{1}{2}}.
\end{equation}
\end{rem}

\subsubsection{Mixed finite element approximation of the Poisson equation}\label{sec:RTPoisson}

We now consider a mixed finite element approximation of the Poisson problem \eqref{approxPoisson} with the one dimensional version of Raviart-Thomas elements. More precisely, the mixed finite element spaces in 1D turn out to be $(W_h^{k+1},V_h^k)$ finite element spaces, where 
\[W_h^{k+1}:=\{u\in C^0(0,L):u_{|I_i}\in P_{k+1}(I_i),\,\forall i\},\]
and $V_h^k$ is defined by \eqref{defVhk}. 

For $k\geq 0$, we look for a couple $(E_{N_H,h},\Phi_{N_H,h})\in W_h^{k+1}\times V_h^k$ such that for all $(\eta,\zeta)\in W_h^{k+1}\times V_h^k$,
\begin{equation}
\label{RTps}
\left\{
\begin{array}{l}
\ds \int_0^L \Phi_{N_H,h}\,\eta'\,\dD x =\int_0^L E_{N_H,h}\,\eta\,\dD x,\\ \, \\
\ds\int_0^L \partial_x E_{N_H,h} \,\zeta\, \dD x = \int_0^L (C_{0,h}-\rho_0)\,\zeta\,\dD x.
\end{array}
\right.
\end{equation}

With this choice of discretization for the Poisson equation, we are not able to prove the conservation of total energy. By contrast, it is possible to obtain the control of $\|E_{N_H,h}(t)\|_\infty$ in this case, which gives a positive lower bound for $\alpha(t)$ and then allows to deduce a stability result in $L^2(0,L)$ for the coefficients $C_{n,h}$, from the stability in weighted $L^2$ norm for $f_{N_H,N_x}$ stated in Proposition \ref{prop:stab_L2dvdx}. 

\begin{thm}
  For any $N_H\geq 3$, we consider the solution $(C_{n,h},E_{N_H,h},\Phi_{N_H,h})$ to the system \eqref{dgcn}--\eqref{lfflx} together with  \eqref{RTps}. Then, we have
  \begin{itemize}
  \item Conservation of mass.
  \item Furthermore,  assuming that $\|f_{N_H,N_x}(0)\|_{\omega(0)}<+\infty$,  we get
    $$
\|f_{N_H,N_x}(t)\|_{\omega(t)}\leq \|f_{N_H,N_x}(0)\|_{\omega(0)}\,e^{t/4\gamma},\quad \forall t\geq 0.
$$
\item Finally, there exists a constant $C_T>0$ such that \[\|E_{N_H,h}(t)\|_{\infty}\leq C_T, \quad \forall t\in [0,T].\]
\end{itemize}
\end{thm}

\begin{proof}
First and second items follows from Proposition \ref{prop:stab_L2dvdx}. Then, we follow the same guidelines as in the proof of Proposition \ref{prop:controlE}. By Sobolev and Poincar\'e-Wirtinger inequalities, we have
\[ \|E(t)\|_{\infty}^2\leq c\, \|E(t)\|^2_{H^1}\leq c' \,\|\partial_x E(t)\|_2^2.\]
Since $E \in W_h^{k+1}$, we have that $\partial_x E \in V_h^k$. Then we take $\zeta=\partial_x E$ as test function in the second equation of \eqref{RTps} to get
\[ \|\partial_x E(t)\|_2^2 = \int_0^L C_0\,\partial_x E\,\dD x \leq \|\partial_x E(t)\|_2\,\|C_0(t)\|_2,  \]
and then
\[\|E(t)\|_{\infty}^2\leq c'\,\|C_0(t)\|_2^2.\]
Finally we proceed exactly in the same way as in the proof of Proposition \ref{prop:controlE}, using the definition \eqref{defalphaDG1} of $\alpha$ and the weighted $L^2$ stability result stated in Proposition \ref{prop:stab_L2dvdx}.
\end{proof}

\begin{rem}
Note that the same result can be established for a conforming approximation of the electric potential, corresponding to a direct integration of the discrete Poisson problem \eqref{approxPoisson}, which is straightforward in 1D.
\end{rem}

%

\section{Time discretization}
\label{sec:4}
\setcounter{equation}{0}
\setcounter{figure}{0}
\setcounter{table}{0}

We  apply a second order Runge-Kutta scheme to the discontinuous Galerkin method for the Vlasov equation \eqref{dgcn}--\eqref{lfflx} with local discontinuous Galerkin approximation of the Poisson equation \eqref{dgP}--\eqref{fluxps}, coupled with a discretization of $\alpha$ solution to \eqref{design_alpha}. 

We denote by $\mathbf C=(C_0,\ldots,C_{N_H-1})$ the solution to \eqref{dgcn}--\eqref{lfflx} and by $(\cdot,\cdot)$ the standard $L^2$ inner product on the space interval $(0,L)$
\[(C_n,\varphi):=\int_0^LC_n\,\varphi\,\dD x.\]
 For the computation of the source terms, we introduce $b_n^j$ defined for $n=0,\ldots,N_H-1$ and $j=1,\ldots,N_x$ by
 \begin{equation*}
b_n^j(\alpha,\mathbf{C},E,\varphi_n)=-\frac{I(\alpha,E)}{\alpha}\int_{I_j}\left(n\,C_n+\sqrt{(n-1)n}\,C_{n-2}\right)\,\varphi_n\,\dD x-\alpha\,\sqrt{n}\int_{I_j}E\,C_{n-1}\,\varphi_n\,\dD x,
\end{equation*}
where we use the expression $I(\alpha,E)$ for $\alpha'$ as defined in \eqref{design_alpha}. We set $a_n=\sum_j\,a_n^j$ and $b_n=\sum_j\,b_n^j$.

Equipped with these compact notations, we can write the semi-discrete discontinuous Galerkin method for the Vlasov equation \eqref{dgcn} as
\begin{equation}\label{dgcn_compact}
 \frac{\dD}{\dD t}(C_n,\varphi_n)+a_n(g_n,\varphi_n)+b_n(\alpha,\mathbf{C},E,\varphi_n)=0,
\end{equation} 
with, according to \eqref{design_alpha}, $\alpha$ satisfying 
\begin{equation*}
\alpha'=I(\alpha,E).
\end{equation*}
Let us underline that $b_0=0$, then equation \eqref{dgcn_compact} for $n=0$ does not depend on the electric field $E$.

We now define the time discretization. Let $\Delta t>0$ be the time step. We compute an approximation $\mathbf{C}^m=(C_0^m,\ldots,C_{N_H-1}^m)$ of the solution $\mathbf{C}$ at time $t^m=m\,\Delta t$ for $m\geq 0$. 
Assuming known $\mathbf{C}^{m}$, we compute $\mathbf{C}^{m+1}$ using the following procedure.

We first apply a classical explicit Euler scheme with a half-time step $\Delta t/2$. 
\begin{itemize}
\item We compute $C_0^{(1)}$ with
\beq
\label{rk-1_0}
\f{2\,(C_0^{(1)}-C_0^{m},\varphi_0)}{\Delta t} \,+\,a_0(g_0^m,\varphi_0)\, =\, 0, \quad \forall \; \varphi_0\,\in\,
V_h^k\,.
\eeq
\item Using $C_0^{(1)}$, we solve the DG approximation of the Poisson equation \eqref{dgP}--\eqref{fluxps} to obtain $E^{(1)}$. We can then define 
$$
E^{m+1/4} = \f12\,\left(E^m+E^{(1)}\right)\,.
$$
\item Finally, we compute $C_n^{(1)}$ for $n=1,\ldots,N_H-1$ thanks to 
\beq
\label{rk-1}
\f{2\,(C_n^{(1)}-C_n^{m},\varphi_n)}{\Delta t} \,+\,a_n(g_n^m,\varphi_n) \,+\,
b_n(\alpha^m,\bC^m,E^{m+1/4},\varphi_n)\, =\, 0, \quad \forall \; \varphi_n\,\in\,
V_h^k\,,
\eeq
with
$$
\frac{2(\alpha^{(1)}- \alpha^m)}{\Delta t} \,=\, I(\alpha^m, E^{m+1/4}). 
$$
\end{itemize}
Then using the same procedure, we compute a second stage with a time step $\Delta t$.
\begin{itemize}
\item We compute $C_0^{m+1}$ with
\begin{equation*}
\f{(C_0^{m+1}-C_0^{m},\varphi_0)}{\Delta t} \,+\,a_0(g_0^{(1)},\varphi_0)\, =\, 0, \quad \forall \; \varphi_0\,\in\,
V_h^k\,.
\end{equation*}
\item Using $C_0^{m+1}$, we solve the DG approximation of the Poisson equation \eqref{dgP}--\eqref{fluxps} to obtain $E^{m+1}$. We can then define 
$$
E^{m+1/2} = \f12\,\left(E^m+E^{m+1}\right)\,.
$$
\item Finally, we compute $C_n^{m+1}$ for $n=1,\ldots,N_H-1$ thanks to 
\beq
\label{rk-3}
\f{(C_n^{m+1}-C_n^{m},\varphi_n)}{\Delta t} \,+\,a_n(g^{(1)}_n,\varphi_n) \,+\,
b_n(\alpha^{(1)},\bC^{(1)},E^{m+1/2},\varphi_n)\, =\, 0, \quad \forall \; \varphi_n\,\in\,
V_h^k\,,
\eeq
with
$$
\frac{\alpha^{m+1}- \alpha^m}{\Delta t} \,=\, I(\alpha^{(1)}, E^{m+1/2}). 
$$
\end{itemize}

This scheme corresponds to the one proposed in \cite{Filbet2020} when the scaling function $\alpha$ is constant. In that case, mass and total energy are preserved. However, in our case since $\alpha$ depends on time, it produces an additional error on the energy conservation. In the next section, we will observe that the variations of $\alpha$ are small, hence the variations of total energy are very limited.  

%

\section{Numerical examples}
\label{sec:5}
\setcounter{equation}{0}
\setcounter{figure}{0}
\setcounter{table}{0}

In this section, we will verify our proposed DG/Hermite Spectral method for the one-dimensional Vlasov-Poisson (VP) system. We take  $N_H$ modes for Hermite spectral basis, and $N_x$ cells in space. We note that due to the Hermite spectral basis, there is no truncation error for the conservation of mass, momentum and energy from cut-off along the $v$-direction. The scaling parameter $\alpha$ is chosen according to the scaling of the initial distribution and the Hou-Li filter with $2/3$ dealiasing rule \cite{hou2007computing,filtered} will be used, if without specification. In all our numerical experiments, we choose $\gamma=0.01$ in such a way that the function $\alpha$ slowly decreases.

In the following, we take $P_2$ piecewise polynomial in space and 2nd order scheme in time. We denote this scheme as ``DG-H". We compute reference solutions using a positivity-preserving PFC scheme proposed in \cite{filbet:01}, which is denoted as ``PFC". The PFC uses discrete velocity coordinate, and the mesh size is $N_x\times N_v$.

First of all, we performed numerical simulations on the Landau damping and obtained similar results as in \cite[Section 4.1]{Filbet2020}. In this case, the quantity $\|\bE(t)\|_\infty \backsim e^{-\kappa_L\,t}$, with $\kappa_L>0$, hence the function
$\alpha$ given by \eqref{eq:def_alpha} decreases slowly.  In the following, we present more challenging numerical tests where the electric field varies with respect to time.

\subsection{Two stream instability}
\label{sec:4.1}

In this example, we consider the two stream instability problem with the initial distribution function
\beq
\label{2stream}
f(t=0,x,v)=\f{2}{7}(1+5v^2)(1+\kappa((\cos(2kx)+\cos(3kx))/1.2+\cos(kx))\f{1}{\sqrt{2\pi}}e^{-v^2/2}\,,
\eeq
where $\kappa=0.01$ and $k=1/2$. For this case, we have
$$
\left\{
  \begin{array}{l}
\ds C_0(t=0,x)\,=\,\f{12}{7}(1+\kappa((\cos(2kx)+\cos(3kx))/1.2+\cos(kx))\,,\\
\ds C_2(t=0,x)\,=\,\f{10\sqrt{2}}{7}(1+\kappa((\cos(2kx)+\cos(3kx))/1.2+\cos(kx))\,.
  \end{array}
\right.
$$
Other $C_n$'s are all zero and $\alpha_0=1$. The background density is $\rho_0=12/7$. The length of the domain in the $x$-direction is $L=4\pi$.

We compute the solution up to time $T=50$. For DG-H, we take $N_x=64$ and $N_H=128$. We show the time evolution of the relative deviations of discrete mass, momentum and total energy in Figure \ref{fig:1} $(a)$. We can see these errors are still up to machine precision. Although the total energy varies a little due to the time discretization, the errors are at the level of $10^{-9}$ which is rather small. We also plot the time evolution of the electric field in $L^2$ norm in Figure \ref{fig:1} $(b)$.

Furthermore, we also present in Figure \ref{fig:2}, the time evolution of weighted $L^2$ norm of  the distribution and  the scaling function $\alpha$ given by \eqref{eq:def_alpha} since it plays a crucial role in the stability analysis of the DG-H method. On the one hand, on the time interval $t \in (0,25)$, the $L^2$ norm of the distribution increases almost exponentially, then it decreases and oscillates. On the other hand the variations of the scaling function $\alpha$  remain small even if the electric field strongly varies with respect to time due to the instability.

Finally, we compare our numerical results with those obtained from the PFC scheme \cite{filbet:01} on a refined  mesh $256\times 1024$. For this case, we show the surface plots of $f$ at $t=20$,  $30$ and $40$ in Figure \ref{fig:3} for both methods. The results are comparable for short time $t\leq 30$, then for larger time, fine structures of the distribution function localized in the center are removed for the DG-H method due to the coarser grid.

\begin{figure}
  \centering
  \begin{tabular}{cc}
	\includegraphics[width=3.2in,clip]{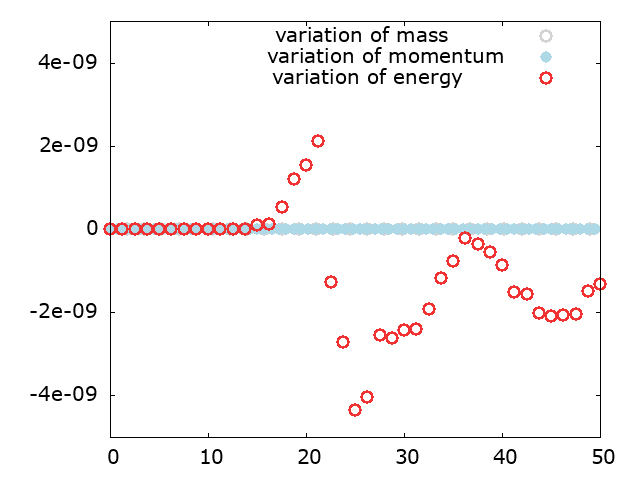}&
	\includegraphics[width=3.2in,clip]{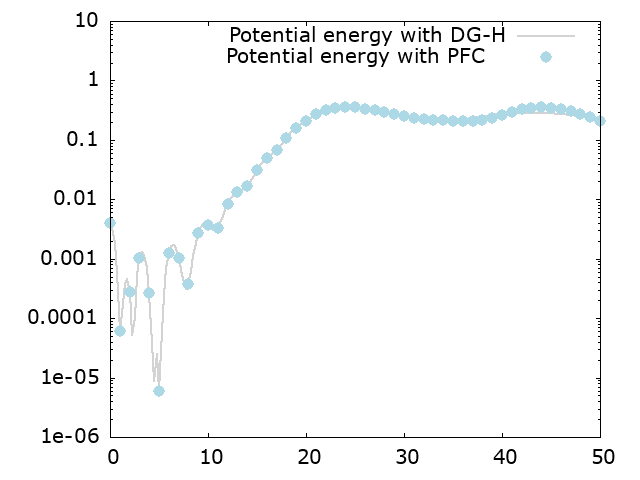}
        \\
(a)&(b)
        \end{tabular}
	\caption{{\bf Two stream instability:} $(a)$ deviation of mass, momentum and energy, $(b)$ time evolution of the electric field in $L^2$ norm  in logarithmic value with DG-H: $N_x\times N_H = 64 \times 128$ and the reference solution is from the PFC scheme with $N_x\times N_v=256\times 1024$.}
	\label{fig:1}
\end{figure}

\begin{figure}
	\centering
\begin{tabular}{cc}
	\includegraphics[width=3.2in,clip]{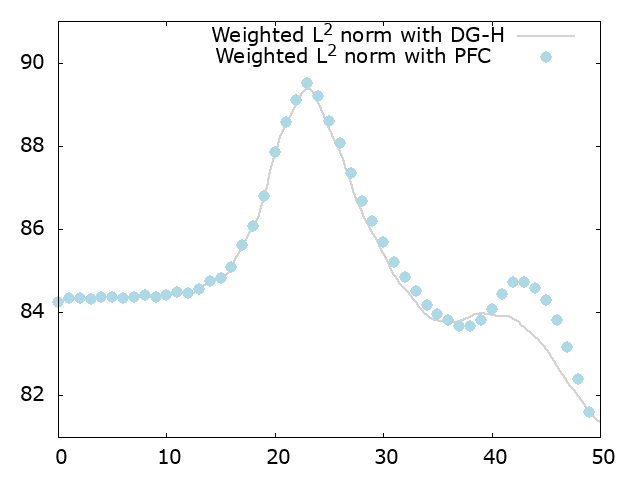}&
	\includegraphics[width=3.2in,clip]{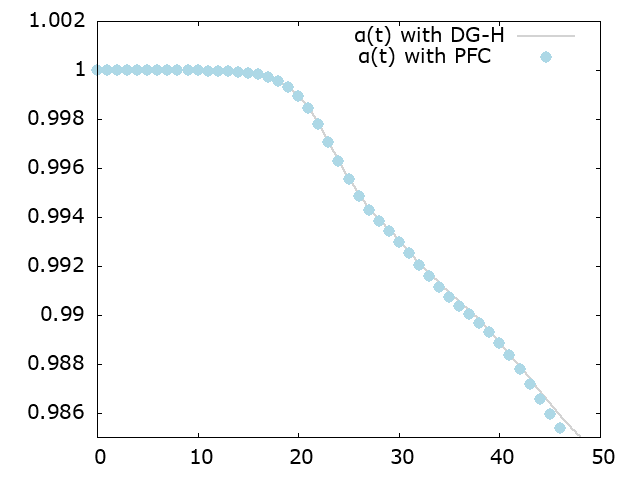}\\
        (a)&(b)
        \end{tabular}
	\caption{{\bf Two stream instability:} $(a)$ time evolution of the weighted $L^2$ norm of $f$, $(b)$ time evolution of the scaling function $\alpha$ for DG-H with $N_x\times N_H = 64 \times 128$ and the reference solution is from the PFC scheme with $N_x\times N_v=256\times 1024$.}
	\label{fig:2}
\end{figure}

\begin{figure}
  \centering
  \begin{tabular}{cc}
	\includegraphics[width=3.2in,clip]{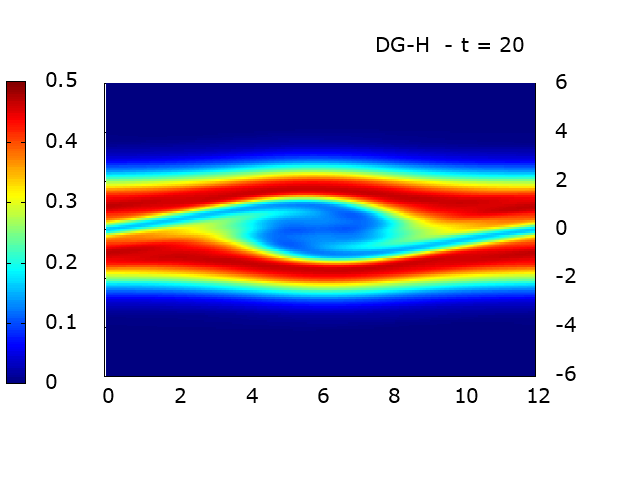} &
	\includegraphics[width=3.2in,clip]{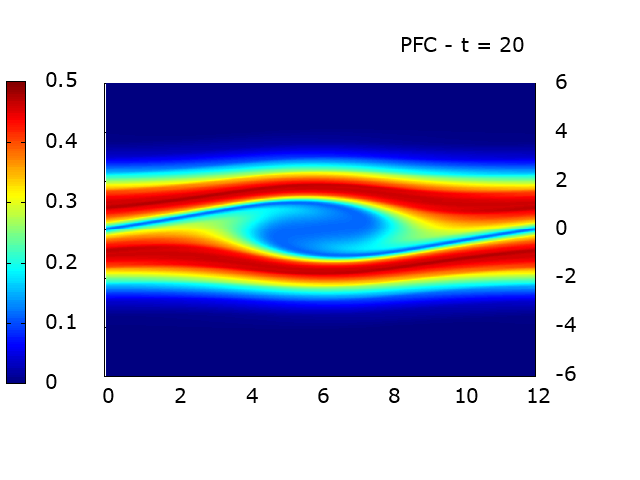}
        \\
        \includegraphics[width=3.2in,clip]{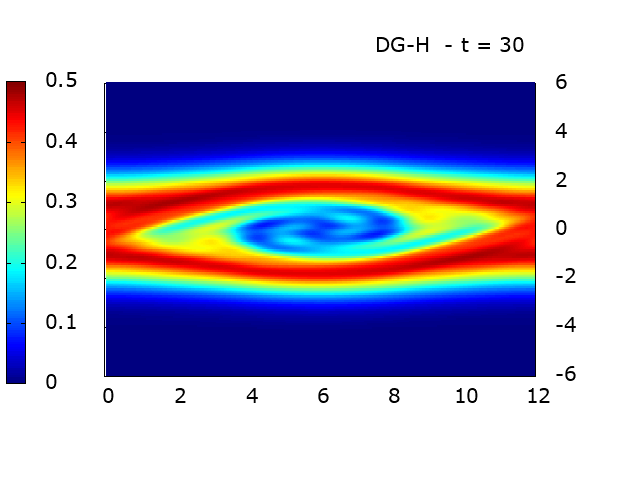} &
	\includegraphics[width=3.2in,clip]{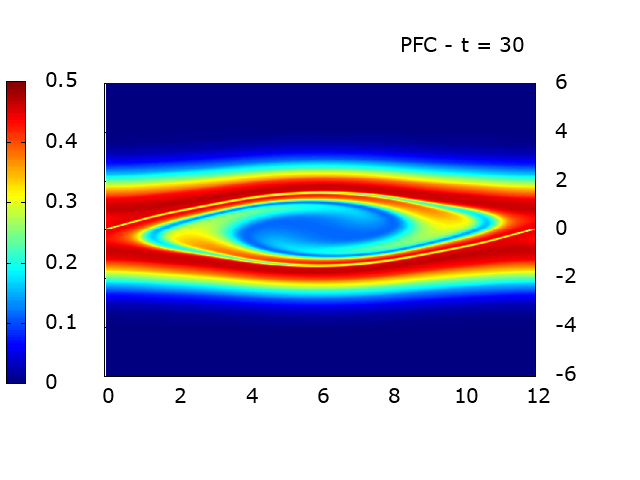}
         \\
        \includegraphics[width=3.2in,clip]{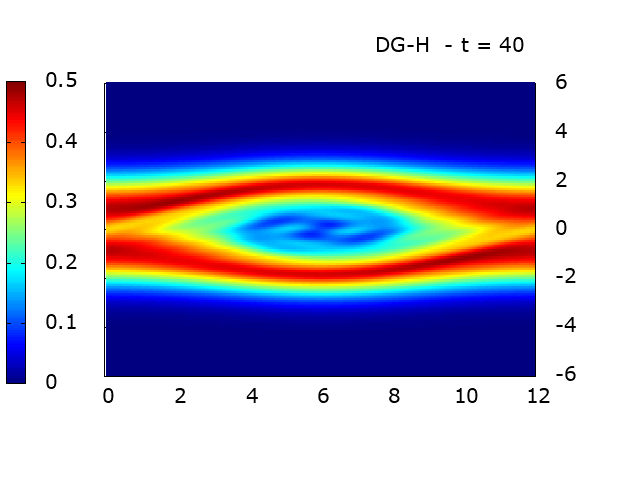} &
	\includegraphics[width=3.2in,clip]{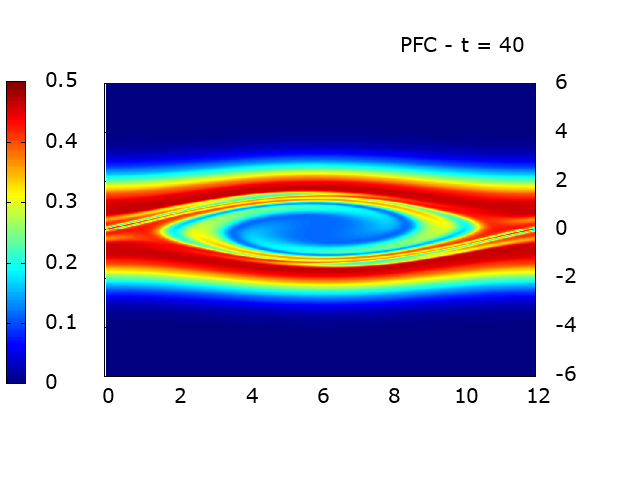}
        \\
        (a) & (b)
\end{tabular}
	\caption{{\bf Two stream instability:}  Surface plot of the distribution function $f$ at $t=20$, $30$ and $40$ with $(a)$ $N_x\times N_H = 64 \times 128$ for DG-H; $(b)$ $N_x\times N_v = 256\times 1024$ for PFC. }
	\label{fig:3}
\end{figure}

\subsection{Bump-on-tail instability}
\label{sec:4:2}
Then we consider the bump-on-tail instability problem with the initial distribution as
\begin{align}
\label{bot}
f(0,x,v)=f_{b}(v)(1+\kappa\cos(k\,n\,x))\,,
\end{align} 
where the bump-on-tail distribution is
\begin{align}
f_{b}(v)=\frac{n_p}{\sqrt{\pi}v_{p}}e^{-v^2/v^2_{p}}+\frac{n_b}{\sqrt{\pi}v_{b}}e^{-(v-v_{d})^2/v_{b}^2}\,.
\end{align}
We choose a strong perturbation with $\kappa=0.04$, $n=3$ and $k=1/10$ and the other parameters are set to be $n_p=0.9$, $n_b=0.1$, $v_{d}=4.5$, $v_{p}=\sqrt{2}$, $v_{b}=\sqrt{2}/2$. The computational domain is $[0,20\pi]\times[-8, 8]$.  These settings have been used in \cite{shoucri1974} and \cite[Section 4.3]{Filbet2020}. For this case, we take the initial scaling function to be   $\alpha_0 =5/7$.

Again, we take $N_x\times N_H=64\times 128$ for DG-H and compare the
solutions to PFC with $N_x\times N_v=256\times 1024$. We first show
the time evolution of the relative deviations of discrete mass,
momentum and total energy in Figure \ref{fig:4} $(a)$. Here, the
errors on mass and total energy are up to machine precision whereas
the momentum varies with respect to time up to $10^{-5}$. We remind
that our space discretization does not ensure conservation of
momentum. We also plot the time evolution of the electric field in
$L^2$ norm in Figure \ref{fig:4} $(b)$. We compare them to the results
obtained by using the PFC method with mesh size $256\times 1024$. We
can see these results have the same structure and they are similar to
those in \cite{shoucri1974}. We also present the time evolution of the
weighted $L^2$ norm and the scaling function $\alpha$ in Figure
\ref{fig:5}.  The $L^2$ norm first increases almost exponentially
fast, hence  it stabilizes for larger time and oscillates, whereas as expected, the scaling function slowly decreases.   

Finally we show the surface plots of the distribution function at $t=12.5$, $25$ and $50$ in Figure \ref{fig:6}. From the comparison of these two methods, we can find that at the beginning $t \leq 20$, the solutions are very close. But as time evolves,
the solutions are moving in different phases. However, the results from DG-H look relatively well.

\begin{figure}
  \centering
  \begin{tabular}{cc}
	\includegraphics[width=3.2in,clip]{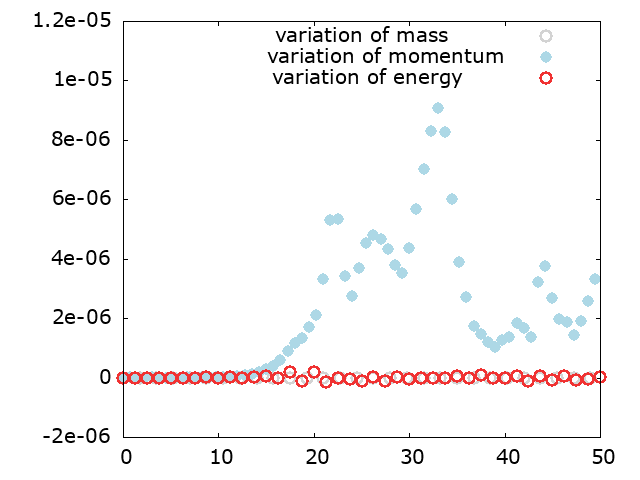}&
	\includegraphics[width=3.2in,clip]{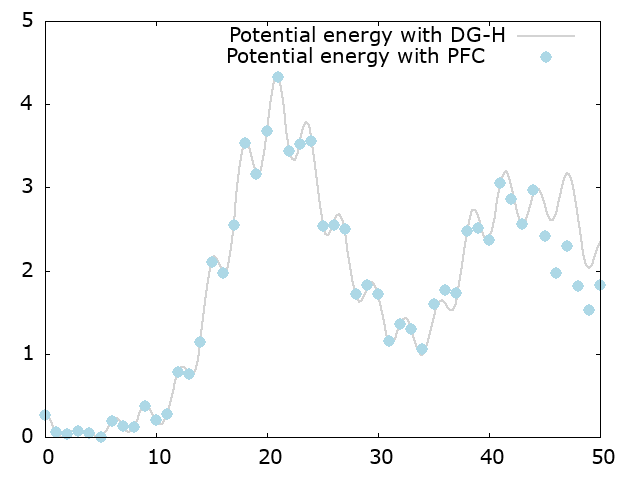}
        \\
        (a) & (b)
        \end{tabular}
	\caption{{\bf Bump-on-tail instability:} $(a)$ deviation of mass, momentum and energy and $(b)$ time evolution of the electric field in $L^2$ norm  in logarithmic value for DG-H: $N_x\times N_H = 64 \times 128$ whereas  the reference solution is obtained from the PFC scheme with $N_x\times N_v=256\times 1024$.}
	\label{fig:4}
\end{figure}

\begin{figure}
  \centering
  \begin{tabular}{cc}
	\includegraphics[width=3.2in,clip]{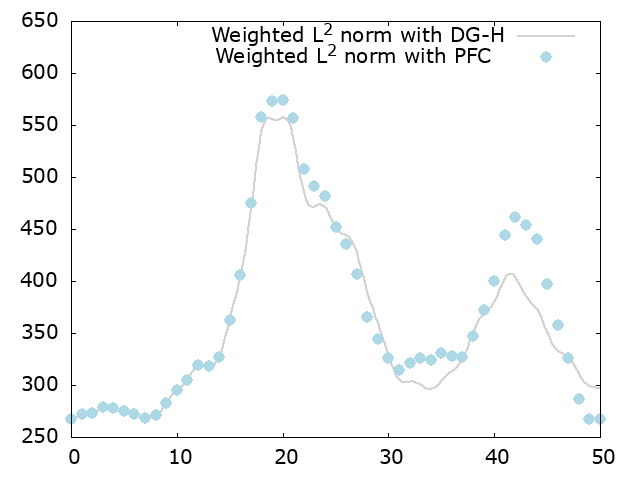} &
	\includegraphics[width=3.2in,clip]{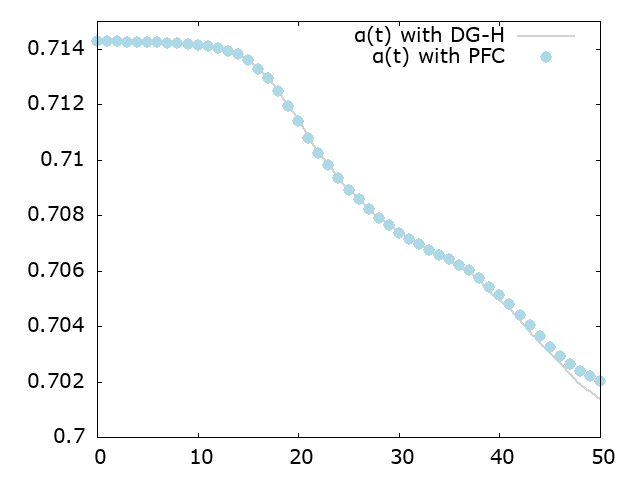}\\
        (a)&(b)
        \end{tabular}
	\caption{{\bf Bump-on-tail instability:} $(a)$ time evolution of the weighted $L^2$ norm of $f$, $(b)$ time evolution of the scaling function $\alpha$ for DG-H with $N_x\times N_H = 64 \times 128$, whereas the reference solution is from the PFC scheme with $N_x\times N_v=256\times 1024$.}
	\label{fig:5}
\end{figure}

\begin{figure}
  \centering
  \begin{tabular}{cc}
	\includegraphics[width=3.2in,clip]{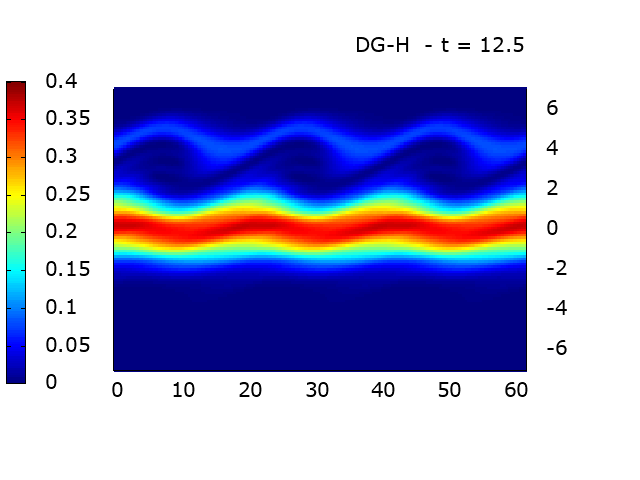} &
	\includegraphics[width=3.2in,clip]{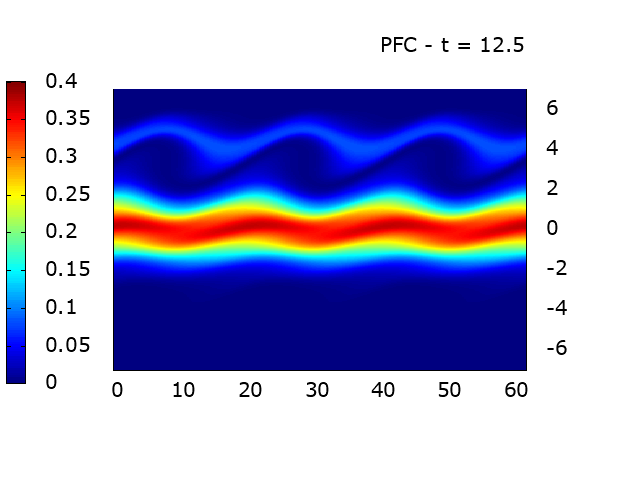}
        \\
        \includegraphics[width=3.2in,clip]{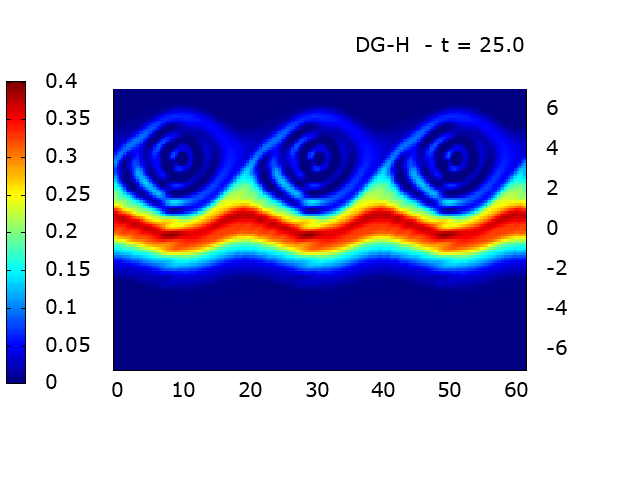} &
	\includegraphics[width=3.2in,clip]{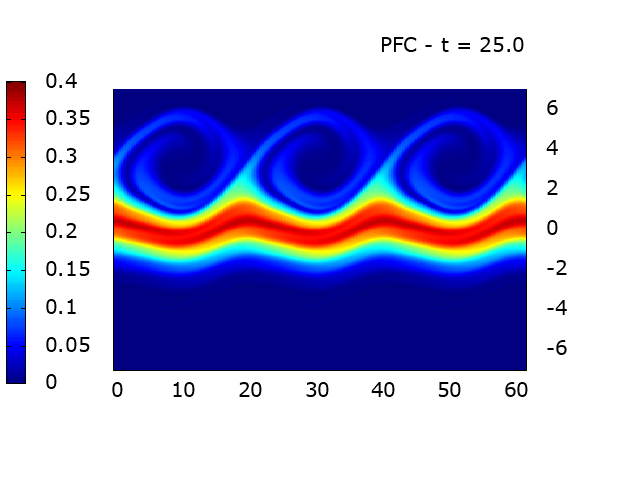}
         \\
        \includegraphics[width=3.2in,clip]{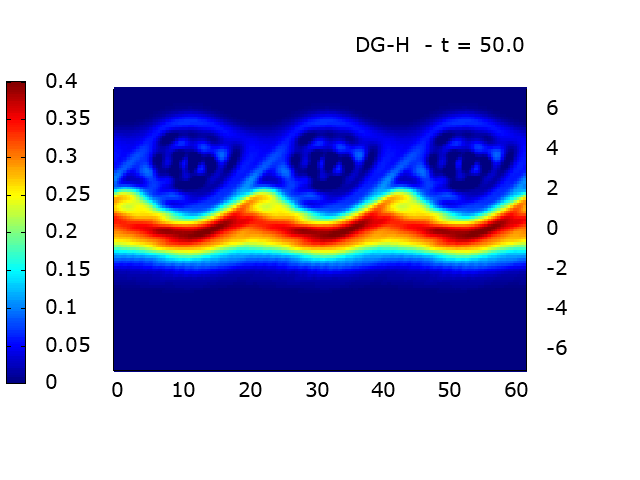} &
	\includegraphics[width=3.2in,clip]{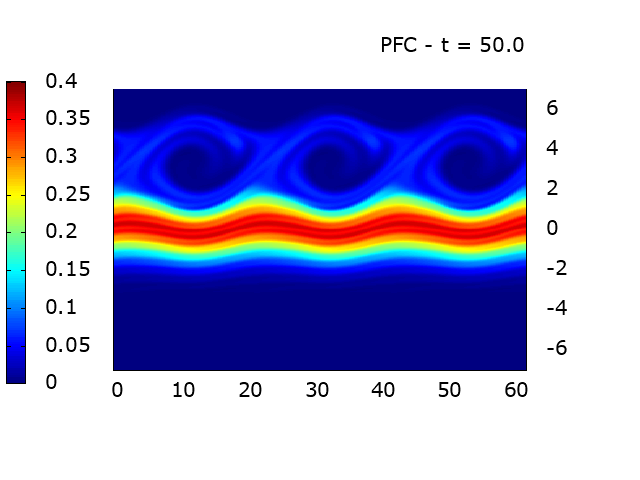}
        \\
        (a) & (b)
\end{tabular}
	\caption{{\bf Bump-on-tail instability:} Surface plot of the distribution function $f$ at $t=12.5$, $25$ and $50$ with $(a)$ $N_x\times N_H = 64 \times 128$ for DG-H and $(b)$ $N_x\times N_v = 256\times 1024$ for PFC. }
	\label{fig:6}
\end{figure}

%

\subsection{Ion acoustic wave}
\label{sec:4:3}
Finally, we  consider a multiscale problem occuring in kinetic plasma physics, that is, the time evolution of an ion acoustic wave where both electrons and ions dynamics are interacting (see \cite{fatone2019}). In this example, the initial distribution of electrons is given by
\begin{align}
\label{iaw}
f_e(0,x,v)\,=\,\frac{1}{\sqrt{2\pi}}\,(1+\kappa\cos(k\,x))\,e^{-|v-v_d|^2/2},
\end{align}
where $v_d=2$ is a drift velocity whereas the initial ditribution of ions is at equilibrium
\begin{align}
f_{i}(0,x,v)= \frac{1}{\sqrt{2\pi v_{th,i}^2}}\,e^{-v^2/2\,v_{th,i}^2}\,.
\end{align}
We choose a small perturbation with $\kappa=0.0001$ and $k=2\pi/10$ and the other parameters are set to be $v_{th,i}=1/50$. The computational domain is $[-5,5]$. We now consider two Vlasov equations for electrons and ions coupled through the Poisson equation
$$
-\partial_{xx} \Phi \,=\,  n_i - n_e.
$$
As in \cite{fatone2019}, we choose a reduced mass ratio ($m_e/m_i =
1/25$). The choice of these parameters has been made to trigger an
ion-acoustic wave instability. Concerning the numerical parameters, we
take $N_x\times N_H=128\times 128$ for DG-H for $f_e$ and $f_i$ with
$\alpha_e(0)=1$ and $\alpha_i(0)=1/v_{th,i}$, hence we compare the
solutions to PFC with $N_x\times N_v=256\times 1024$. We first show
the time evolution of the relative deviations of discrete mass,
momentum and total energy in Figure \ref{fig:7} $(a)$. The errors on
mass and total energy are of order $10^{-10}$ whereas the momentum
varies with respect to time up to $10^{-9}$. We also plot the time
evolution of the electric field in $L^2$ norm in Figure \ref{fig:7}
$(b)$. We compare them to the results by using the PFC method with
mesh size $256\times 1024$. We can see these results have the same
structure and they are similar to those in \cite{fatone2019}. We also
present the time evolution of the weighted $L^2$ norm of the
distributions $f_e$ and $f_i$ in Figure \ref{fig:8}.  On the one hand,
the behavior of the weighted $L^2$ norms of $f_e$ and $f_i$ is quite
different : the quantity $t\mapsto \|f_i(t)\|_{\omega_i}$ first
increases and then  it stabilizes for larger time, whereas $t\mapsto
\|f_e(t)\|_{\omega_e}$ decreases. On the other hand, as it is expected
$\alpha_e$ and $\alpha_i$ decrease slowly.

Finally we show a zoom of the surface plots of the distribution function $f_e$ when the instability develops at time $t=175$, $200$ and $250$ in Figure \ref{fig:9}.  Of course, we get a lower resolution with DG-H on the coarse mesh but when $t\leq 200$, both solutions have the same behavior.  As time evolves $t\geq 200$, the solutions are moving in different phases. However, the results from DG-H look relatively well.

\begin{figure}
  \centering
  \begin{tabular}{cc}
	\includegraphics[width=3.2in,clip]{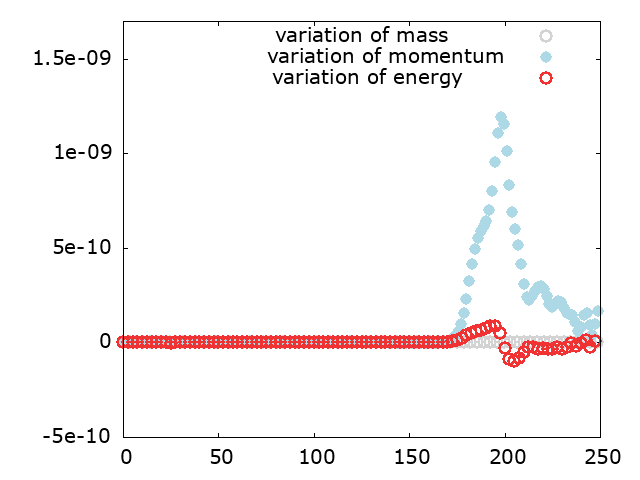}&
	\includegraphics[width=3.2in,clip]{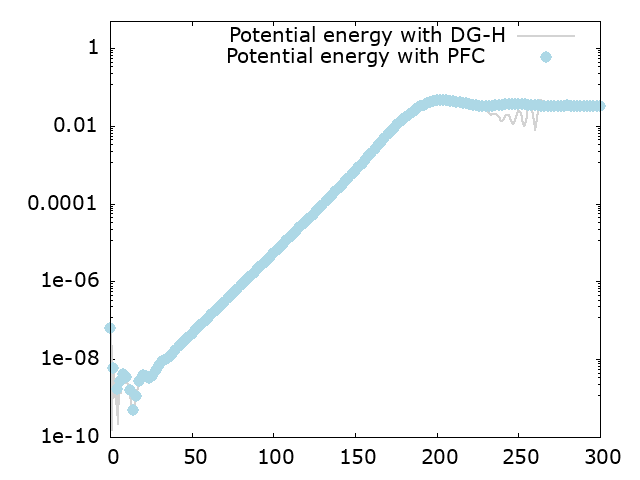}
        \\
        (a) & (b)
        \end{tabular}
	\caption{{\bf Ion acoustic wave:} $(a)$ deviation of total mass, momentum and energy and $(b)$ time evolution of the electric field in $L^2$ norm  in logarithmic value for DG-H: $N_x\times N_H = 64 \times 128$ whereas  the reference solution is obtained from the PFC scheme with $N_x\times N_v=256\times 1024$.}
	\label{fig:7}
\end{figure}

\begin{figure}
  \centering
  \begin{tabular}{cc}
	\includegraphics[width=3.2in,clip]{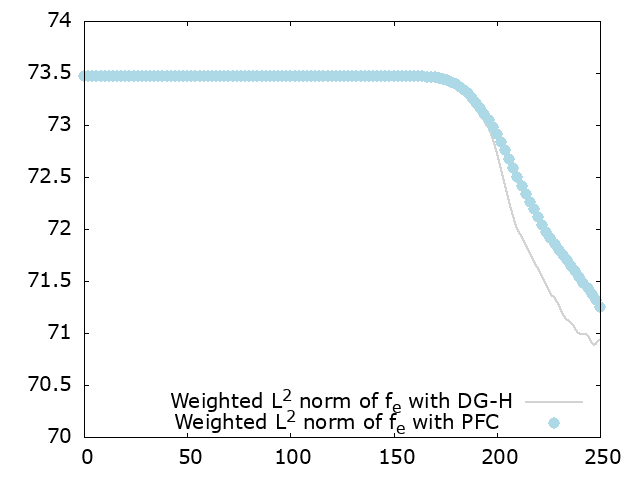} &
	\includegraphics[width=3.2in,clip]{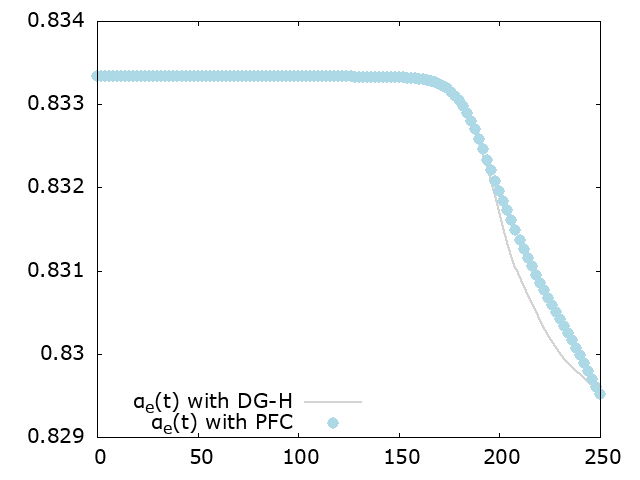}\\
        \includegraphics[width=3.2in,clip]{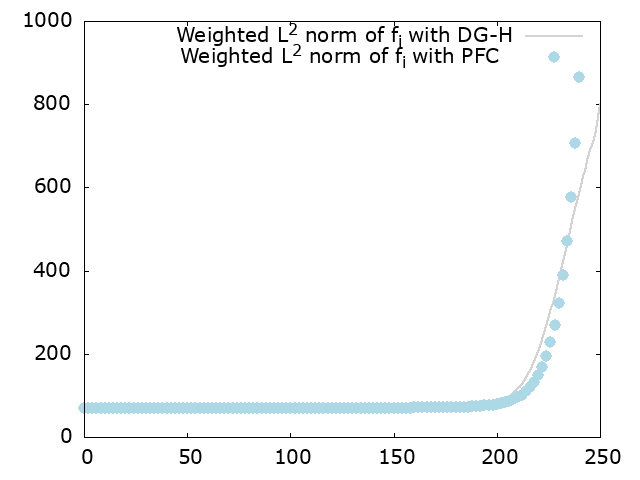} &
	\includegraphics[width=3.2in,clip]{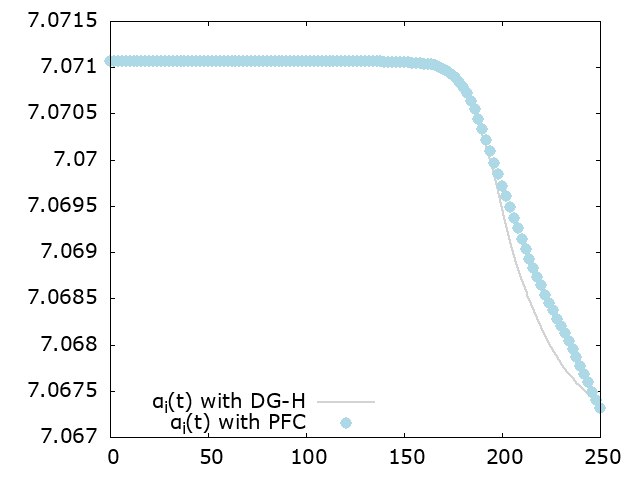}\\
        (a)&(b)
        \end{tabular}
	\caption{{\bf Ion acoustic wave:} $(a)$ time evolution of the weighted $L^2$ norm of  $f_e$ (top) and  $f_i$ (bottom) and  $(b)$ time evolution of  $\alpha_e$ (top) and  $\alpha_i$ (bottom) for DG-H with $N_x\times N_H = 64 \times 128$, whereas the reference solution is from the PFC scheme with $N_x\times N_v=256\times 1024$.}
	\label{fig:8}
\end{figure}

\begin{figure}
  \centering
  \begin{tabular}{cc}
	\includegraphics[width=3.2in,clip]{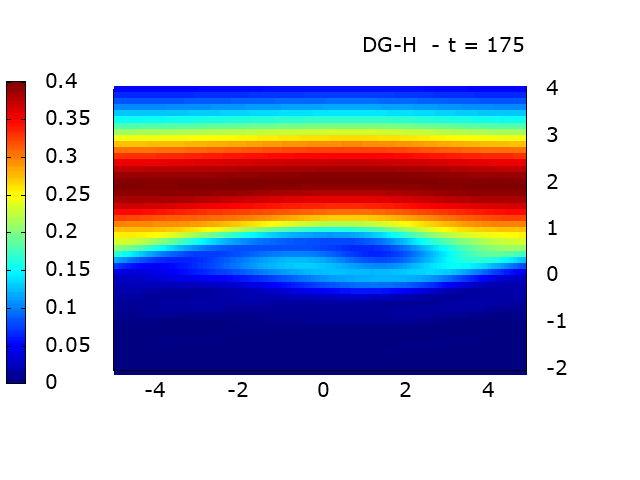} &
	\includegraphics[width=3.2in,clip]{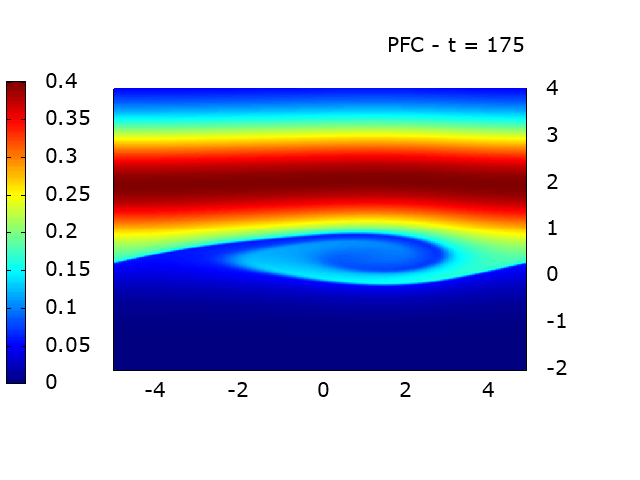}
        \\
        \includegraphics[width=3.2in,clip]{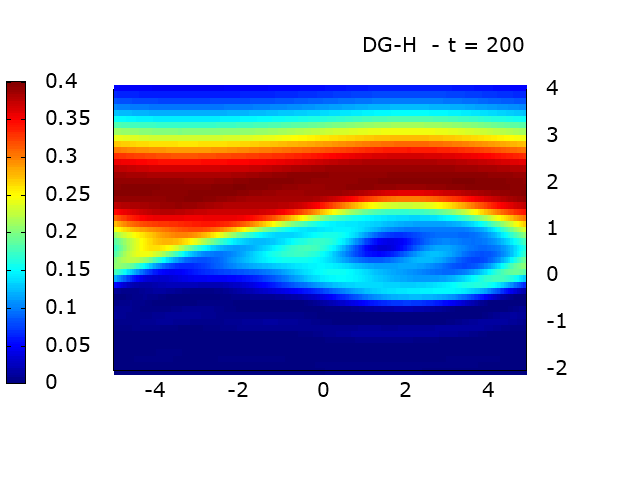} &
	\includegraphics[width=3.2in,clip]{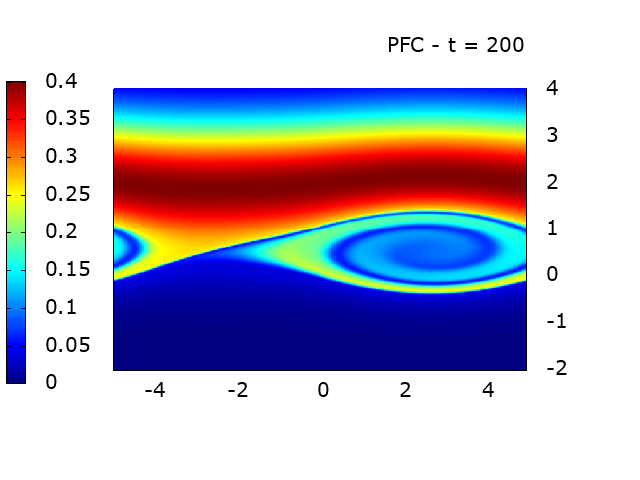}
         \\
        \includegraphics[width=3.2in,clip]{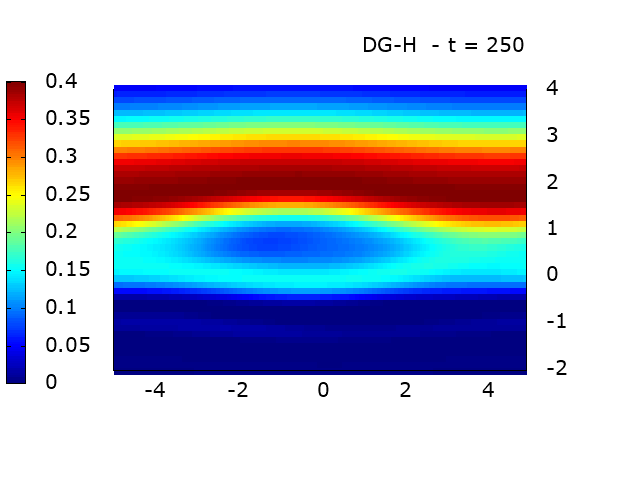} &
	\includegraphics[width=3.2in,clip]{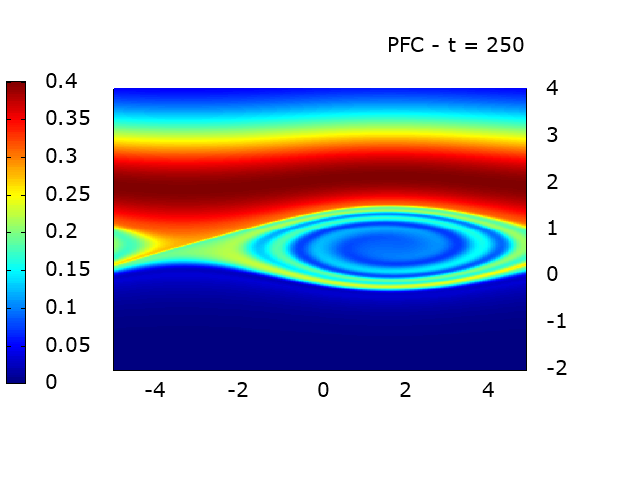}
        \\
        (a) & (b)
\end{tabular}
	\caption{{\bf Ion acoustic wave:} Surface plot of the distribution function $f_e$ at $t=12.5$, $25$ and $50$ with $(a)$ $N_x\times N_H = 64\times 128$ for DG-H and $(b)$ $N_x\times N_v = 256\times 1024$ for PFC. }
	\label{fig:9}
\end{figure}

%

\section{Conclusion and perspectives}

We propose in this paper a spectral Hermite discretization of the Vlasov-Poisson system with a time-dependent scaling factor allowing to prove some stability properties of the numerical solution. It appears that the control of this scaling factor, and more precisely a positive lower bound, is crucial to ensure completely the stability of the method. This property, as well as conservation features, is discussed for several spatial discretizations. Briefly speaking, the scaling factor does not affect the precision of the method compared to the recent work in \cite{Filbet2020}  and guarantees the stability of the numerical approximation.

The present work is the first stone to investigate the convergence
analysis of the proposed symmetric Hermite spectral discretization as
it has already been done for asymmetric Hermite spectral
discretization in \cite{Manzini2017}. Furthermore, the time
discretization proposed in this paper follows the ideas of
\cite{Filbet2020}, but unfortunately it does not constraint energy conservation and $L^2$-stability. One approach would be to construct a Crank-Nicolson type scheme for $f$ and $\alpha$ in such a way that stability properties may be proved rigorously for a full discretized method.

%
\section*{Acknowledgement}

Marianne Bessemoulin-Chatard is partially funded by the Centre Henri Lebesgue (ANR-11-LABX-0020-01) and ANR Project MoHyCon (ANR-17-CE40-0027-01). Francis Filbet  is partially funded by the ANR Project Muffin (ANR-19-CE46-0004) and by the EUROfusion Consortium and has
received funding from the Euratom research and training programme
2019-2020 under grant agreement No 633053. The views and opinions expressed herein do not necessarily reflect those of the European Commission..

\bibliographystyle{abbrv}
\bibliography{refer}
\end{document}